\documentclass[graybox]{svmult}

\usepackage{type1cm}        
\usepackage{graphicx}        
\usepackage{multicol}        
\usepackage[bottom]{footmisc}

\usepackage{newtxtext}       %
\usepackage{newtxmath}       
\usepackage[hmargin=3.cm,bmargin=5.5cm,tmargin=3.cm,headsep=25pt,left=4.cm,rmargin=4.2cm]{geometry}
\newcommand{\simone}[1]{{\color{black}#1}}

\newcommand{\nfun}[2]{N_{#2}(#1)}
\newcommand{\nvec}[2]{N^{\nabla}_{#2}(#1)}
\def\Afr{\mathfrak{A}}
\def\dfr{\mathfrak{d}}

\newcommand{\Xscr}{\mathcal{X}}

\newcommand{\Sph}{{\mathcal S}}

\newcommand{\A}{\mathcal A}

\newcommand{\tond}[1]{{\left(#1\right)}}
\newcommand{\quadr}[1]{{\left[#1\right]}}

\newcommand{\graff}[1]{\{#1\}}

\newcommand{\Poi}[2]{\{#1,#2\}}

\newcommand{\norm}[1]{\left\|#1\right\|}

\newcommand{\Rscr}{\mathcal{R}}
\newcommand{\Pscr}{\mathcal{P}}
\newcommand{\Oscr}{\mathcal O}
\newcommand{\Zscr}{\mathcal Z}

\newcommand{\Orb}{{\mathsf O}}

\renewcommand{\leq}{\leqslant}
\renewcommand{\geq}{\geqslant}

\def\CC{\mathbb{C}}
\def\RR{\mathbb{R}}
\def\ZZ{\mathbb{Z}}
\def\NN{\mathbb{N}}

\def\eps{\epsilon}

\def\im{i}
\def\Chi{\Xscr}
\def\ph{\varphi}

\begin{document}

\title*{Existence and stability of Klein--Gordon breathers in the small-amplitude limit}
\titlerunning{Existence and stability of Klein--Gordon breathers}

\author{D.E. Pelinovsky, T. Penati, and S. Paleari}


\institute{
  D.E. Pelinovsky \at Department of Mathematics, McMaster University, Hamilton, Ontario,
    Canada, L8S 4K1
  \email{dmpeli@math.mcmaster.ca}
  \and
  T. Penati \at Department of Mathematics ``F.Enriques'', Milano University, via Saldini 50, Milano, Italy, 20133
  \email{tiziano.penati@unimi.it}
  \and
  S. Paleari \at Department of Mathematics ``F.Enriques'', Milano University, via Saldini 50, Milano, Italy, 20133
  \email{simone.paleari@unimi.it}
}

%
%
\maketitle

\abstract*{We consider a discrete Klein--Gordon (dKG)
  equation on $\ZZ^d$ in the limit of the discrete nonlinear
  Schr\"{o}dinger (dNLS) equation, for which small-amplitude breathers
  have precise scaling with respect to the small coupling strength
  $\eps$.  By using the classical Lyapunov--Schmidt method, we show
  existence and linear stability of the KG breather from existence and
  linear stability of the corresponding dNLS soliton. Nonlinear
  stability, for an exponentially long time scale of the order
  $\mathcal{O}(\exp(\eps^{-1}))$, is obtained via the normal form
  technique, together with higher order approximations of the KG
  breather through perturbations of the corresponding dNLS soliton.
}

\abstract{We consider a discrete Klein--Gordon (dKG)
  equation on $\ZZ^d$ in the limit of the discrete nonlinear
  Schr\"{o}dinger (dNLS) equation, for which small-amplitude breathers
  have precise scaling with respect to the small coupling strength
  $\eps$.  By using the classical Lyapunov--Schmidt method, we show
  existence and linear stability of the KG breather from existence and
  linear stability of the corresponding dNLS soliton. Nonlinear
  stability, for an exponentially long time scale of the order
  $\mathcal{O}(\exp(\eps^{-1}))$, is obtained via the normal form
  technique, together with higher order approximations of the KG
  breather through perturbations of the corresponding dNLS soliton.
}

\section{Introduction}

Nonlinear oscillators with weak linear couplings on the
$d$-dimensional cubic lattice are described by the discrete
Klein--Gordon (dKG) equation
\begin{equation}
\label{KG}
\ddot{u}_n + V'(u_n) = \eps(\Delta u)_n, \quad n \in \ZZ^d,
\end{equation}
where $\epsilon>0$ is the small coupling strength, $\Delta$ is the
discrete Laplacian operator on $\ell^2(\ZZ^d)$, and $V(u)$ is a
nonlinear potential for each oscillator. The total energy of the
nonlinear oscillators conserves in time $t$ and is given by the
Hamiltonian function
\begin{equation}
\label{Ham}
H(u) = \frac{1}{2} \sum_{n \in \ZZ^d} \dot{u}_n^2 + \frac{\eps}{2} \sum_{n \in \ZZ^d}
\sum_{|k-n|=1} (u_k - u_n)^2 + \sum_{n \in \ZZ^d} V(u_n).
\end{equation}
For illustrative purposes, we deal with a hard anharmonic potential in the form
\begin{equation}
\label{potential-V}
V(u) = \frac12u^2 + \frac1{2+2p}u^{2+2p},
\end{equation}
where $p \in \mathbb{N}$ is assumed for analyticity of the vector
field.  There exists the unique global solution to the Cauchy problem for the dKG
equation (\ref{KG}) with (\ref{potential-V}) equipped with the initial
datum $(u,\dot{u}) \in \ell^2(\ZZ^d) \times \ell^2(\ZZ^d)$, where $u$
stands for $\{ u_n \}_{n \in \ZZ^d}$. Because our main results are
formulated for small initial datum $(u,\dot{u})$, most of the results are applicable for general
anharmonic potentials expanded as
\begin{equation}
\label{potential-expansion}
V(u) = \frac12u^2 + \alpha_p u^{2+2p} + \mathcal{O}(u^{4+2p}) \quad \mbox{\rm as} \quad u \to 0,
\end{equation}
if $\alpha_p \neq 0$. The general anharmonic potential $V$ is
classified as {\em soft} if $\alpha_p < 0$ and {\em hard} if $\alpha_p
> 0$.

{\em Discrete breathers} are time-periodic solutions localized on the
lattice.  Such solutions can be constructed asymptotically by
exploring the two opposite limit: the anti-continuum limit $\epsilon
\to 0$ of weak coupling between the oscillators \cite{MA94} and the
continuum limit $\epsilon \to \infty$ of strong coupling
\cite{BamPP10}.  Compared to these asymptotic approximations, we explore here a different
limit of the dKG equation to the discrete nonlinear Schr\"{o}dinger (dNLS) equation, where
the weak coupling between the oscillators is combined together with
small amplitudes of each oscillator. To be precise, we assume the
scaling
\begin{equation}
\label{e.scale.orig}
  u_n = \eps^{1/2p}\tilde u_n, \quad n \in \ZZ^d
\end{equation}
and rewrite the dKG equation (\ref{KG}) with the potential (\ref{potential-V})
in the perturbed form:
\begin{equation}
  \label{e.KG.scaled}
\ddot{u}_n + u_n + \eps u_n^{1+2p} = \eps (\Delta u)_n, \quad n \in \ZZ^d,
\end{equation}
where the tilde notations have been dropped. By using a formal expansion 
$$u_n(t) = a_n(\eps t) e^{it} + \bar{a}_n(\eps t) e^{-it} + \mathcal{O}(\eps),$$
the following dNLS equation for the complex amplitudes is derived from the requirement
that the correction term $\mathcal{O}(\eps)$ remains bounded
in $\ell^2(\ZZ^d)$ on the time scale of $\mathcal{O}(\eps^{-1})$:
\begin{equation}
\label{dnls-introduction}
2 i a'_n + \gamma_p |a_n|^{2p} a_n = (\Delta a)_n, \quad n \in \ZZ^d,
\end{equation}
where the prime denotes the derivative with respect to the slow time variable $\tau = \eps t$ and
the numerical coefficient $\gamma_p$ is given by
\begin{equation}
\label{gamma-p}
\gamma_p = \binom{1+2p}{1+p}  = \frac{(2p+1)!}{p! (p+1)!}.
\end{equation}
The asymptotic relation between the dKG equation (\ref{e.KG.scaled})
and the dNLS equation (\ref{dnls-introduction}) was observed first in
\cite{PelS12} and was made rigorous by using two equivalent analytical
methods in our previous work \cite{PPP16}.

Discrete breathers of the dKG equation (\ref{KG}) are approximated by
discrete solitons (standing localized waves) of the dNLS equation
(\ref{dnls-introduction}) in the form $a_n(\tau) = A_n e^{-\frac{i}{2}
  \Omega \tau}$, where the time-independent amplitudes satisfies the
stationary dNLS equation
\begin{equation}
\label{stat-dnls-introduction}
\Omega A_n + \gamma_p |A_n|^{2p} A_n = (\Delta A)_n, \quad n \in \ZZ^d.
\end{equation}
The elementary staggering transformation
\begin{equation}
\label{stag-transf}
A_n = (-1)^n \tilde{A}_n, \quad
\Omega = -4d - \tilde{\Omega}
\end{equation}
relates the defocusing version (\ref{stat-dnls-introduction}) to the
focusing version
\begin{equation}
\label{stat-dnls-focusing}
(\Delta \tilde{A})_n + \gamma_p |\tilde{A}_n|^{2p} \tilde{A}_n = \tilde{\Omega} \tilde{A}_n, \quad n \in \ZZ^d.
\end{equation}
In the recent past, existence and stability of discrete solitons in the focusing version
(\ref{stat-dnls-focusing}) has been studied in many details depending
on the exponent $p$ and the dimension $d$. Various approximations of
discrete solitons of the dNLS equation (\ref{stat-dnls-focusing}) are
described in \cite{James}.  Let us review some relevant results on
this subject.

The stationary dNLS equation (\ref{stat-dnls-focusing}) is the
Euler--Lagrange equation of the constrained variational problem
\begin{equation}
\label{constrained-var-pr}
\mathcal{E}_{\nu} = \inf_{a \in \ell^2(\ZZ^d)} \left\{ E(a) : \;\; N(a) = \nu \right\},
\end{equation}
where
\begin{equation}
E(a) = \sum_{n \in \ZZ^d} \sum_{|k-n|=1} |a_{k} - a_n|^2 - \frac{1}{p+1} \sum_{n \in \ZZ^d} |a_n|^{2p+2}
\end{equation}
is the conserved energy, $N(a) = \sum_{n \in \ZZ^d}
|a_n|^2$ is the conserved mass, and $\nu > 0$ is fixed. The
existence of a ground state as a minimizer of the constrained
variational problem (\ref{constrained-var-pr}) was proven in Theorem
2.1 in \cite{Wei99} for every $\mathcal{E}_{\nu} < 0$. By Theorem 3.1
in \cite{Wei99}, if $p<\frac{2}{d}$, the ground state exists for every
$\nu > 0$, however, if $p \geq \frac{2}{d}$, there exists an
excitation threshold $\nu_d > 0$ and the ground state only
exists for $\nu > \nu_d$.

Variational and numerical approximations for $d = 1$ were employed to
analyze the structure of discrete solitons of the stationary dNLS
equation (\ref{stat-dnls-focusing}) near the critical case $p = 2$
\cite{Mal1,Mal2}. It was shown for single-pulse solitons that although
the dependence $\Omega \mapsto \nu$ is monotone for $p = 1$, it
becomes non-monotone for $p \gtrapprox 1.5$ covering the whole range
$\nu > 0$ for $p < 2$ and featuring the excitation threshold
for $p \geq 2$. Further analytical estimates on the excitation
threshold in the stationary dNLS equation were developed in
\cite{Kar1,Kar2,Kar3}.

Spectral stability of discrete solitons in the dNLS equation (\ref{dnls-introduction})
was analyzed in the limit $\Omega \to \infty$, which can be recast as the anti-continuum limit
of the dNLS equation. It was shown for $d = 1$ in \cite{PKF05,PelS11}
(see Section 4.3.3 in \cite{Pelin2011}) that the single-pulse solitons
are stable in the limit $\Omega \to \infty$ for every $p \in
\mathbb{N}$.  Asymptotic stability of single-pulse solitons for $d =
1$ and $p \geq 3$ was also proven in the same limit in
\cite{Bambusi13} after similar asymptotical stability results were
obtained for small solitons of the dNLS equation in the presence of a
localized potential \cite{CT09,KPS09}.

Spectral and orbital stability of single-pulse discrete solitons in
the dNLS equation (\ref{dnls-introduction}) is determined by the
monotonicity of the dependence $\Omega \mapsto \nu$ according to the
Vakhitov--Kolokolov criterion \cite{KapProm2013,Pelin2011}.
It was shown in \cite{KCP16} that this criterion is related to a
similar energy criterion for spectral stability of discrete breathers
in the dKG equation (\ref{KG}). If $\omega$ is a frequency of the
discrete breathers and $H$ is the value of their energy, then the monotonicity of the
dependence $\omega \mapsto H$ is related to the monotonicity of the
dependence $\Omega \mapsto \nu$ in the dNLS limit.  Further results on
the energy criterion for spectral stability of discrete breathers in the dKG equation
(\ref{KG}) are given in \cite{V1,V2}. In spite of many convincing numerical
evidences, the orbital stability of single-pulse discrete breathers is
still out of reach in the energy methods.

The purpose of this paper is to make precise the correspondence
between existence and linear stability of discrete breathers in the
dKG equation (\ref{KG}) and discrete solitons in the dNLS equation
(\ref{dnls-introduction}). This work clarifies applications mentioned
in Section 4 of our previous paper \cite{PPP16}.  We show how the
Lyapunov--Schmidt reduction method can be employed equally well to
study existence and linear stability of small-amplitude discrete
breathers near the point of their bifurcation from the dNLS limit
under reasonable assumptions on existence and linear stability of the
discrete solitons of the dNLS equation. We also show how normal form
methods (see \cite{BamG93,Bam96,BGG89,PalP16,PalP16a}), combined with
the Lyapunov-Schmidt reduction, are implemented to provide higher
order approximation of the discrete breathers, in the same dNLS
limit. These results represent a considerable improvement with respect
to the corresponding results in~\cite{PalP16}. Long-time nonlinear
stability of small-amplitude discrete breathers then follows, assuming
the discrete soliton of the stationary dNLS equation \eqref{stat-dnls-focusing}
is a ground state of the variational problem \eqref{constrained-var-pr}.

The remainder of this paper consists of three sections.  Section 2
proves the existence of discrete breathers obtained via the
Lyapunov--Schmidt decomposition.  Section 3 describes the linear
stability results obtained by the extension of the same technique.
Section 4 gives the normal form arguments towards the long-time
nonlinear stability of small-amplitude breathers.

\section{Existence via Lyapunov-Schmidt decomposition}

Breathers are $T$-periodic solutions of the dKG equation (\ref{KG})
localized on the lattice. One can consider such strong solutions of
the dKG equation (\ref{KG}) in the space $u \in H^2_{\rm
  per}([0,T];\ell^2(\ZZ^d))$. By scaling the time variable as $\tau =
\omega t$ with $\omega = 2 \pi / T$, it is convenient to consider
$2\pi$-periodic solutions $U \in H^2_{\rm
  per}([-\pi,\pi];\ell^2(\ZZ^d))$ with parameter $\omega$, such that
$u(t) = U(\omega t)$. Breather solutions can be equivalently
represented by the Fourier series
\begin{equation}
  \label{FSsolution}
U(\tau) = \sum_{m\in\ZZ} A^{(m)} e^{\im m \tau}.
\end{equation}
Since $U$ is real, the complex-valued Fourier coefficients satisfy the
constraints:
\begin{equation}
\label{property-1}
A^{(m)} = \overline{A^{(-m)}}\ , \quad m \in \ZZ.
\end{equation}
If the periodic solution has zero initial velocity, i.e., $U'(0) = 0$,
then it follows from reversibility of the dKG equation
(\ref{KG}) in time\footnote{Given a solution $\gamma
  :=\graff{u(t),\dot{u}(t)}$ to the dKG equation (\ref{KG}), another solution is $\tilde\gamma =
  \graff{\tilde u(t)=u(-t),\widetilde{\dot{u}}(t) = -\dot{u}(-t)}$.
  If $\gamma$ is a periodic solution with initial zero velocity, then
  the same is true for $\tilde\gamma$, and since the two solutions
  have the same initial configuration $u(0) = \tilde u(0)$, they are
  solutions of the same Cauchy problem, hence they coincide.} that the
periodic solution is even in time, which implies
\begin{equation}
\label{property-2}
A^{(m)} = A^{(-m)}\ ,\quad  m \in \ZZ.
\end{equation}
As a consequence of the two symmetries, the Fourier coefficients are real,
hence the representation (\ref{FSsolution}) becomes Fourier cosine series
with real-valued coefficients:
\begin{equation}
\label{FS-cosine}
U(\tau) = A^{(0)} + 2 \sum_{m\in\NN} A^{(m)} \cos(m \tau).
\end{equation}

After the scaling transformation (\ref{e.scale.orig}), breather solutions to
the scaled dKG equation (\ref{e.KG.scaled}) satisfy the following boundary-value problem:
\begin{equation}
  \label{e.main.scaled}
\omega^2 U'' + U + \eps U^{1+2p} = \eps \Delta U, \quad U \in H^2_{\rm per}([-\pi,\pi];\ell^2(\ZZ^d)),
\end{equation}
\simone{where we use componentwise multiplication of sequences with
the following convention: $(U^{1+2p})_n = (U_n)^{1+2p}$. }
The existence problem (\ref{e.main.scaled}) can be rewritten in
real-valued Fourier coefficients as
\begin{equation}
    \label{e.fou.eq}
(1-m^2\omega^2) A^{(m)} + \frac{\eps}{2\pi} \int_{-\pi}^{\pi} U^{1+2p}(\tau) e^{-\im m\tau} d\tau = \eps \Delta A^{(m)}, \quad m \in \mathbb{N}_0.
\end{equation}

At $\eps=0$, bifurcation of breathers is expected at $\omega_m = 1/m$,
$m \in \mathbb{N}$, from which the lowest bifurcation value $\omega_1
= 1$ gives a branch of \simone{fundamental breathers with one oscillation on the period $[-\pi,\pi]$.} 
If the solution branch $\omega(\eps)$ and $\{ A^{(m)}(\eps) \}_{m \in
  \mathbb{N}} \in \ell^{2,2}(\ZZ;\ell^2(\ZZ^d))$ is parameterized by
$\eps$, then we are looking for the branch of fundamental breathers to
satisfy the limiting conditions:
\begin{equation}
\label{breather-limit}
\lim_{\eps \to 0} \omega(\eps) = 1,  \quad \lim_{\eps \to 0} A^{(1)}(\eps) \neq 0, \quad
\mbox{\rm and} \quad
\lim_{\eps \to 0} A^{(m)}(\eps) = 0, \quad m \neq 1.
\end{equation}
The limiting conditions (\ref{breather-limit}) are not sufficient for
persistence argument.  In order to define uniquely a continuation of
the solution branch in $\eps$, we consider the stationary dNLS
equation in the form:
\begin{equation}
\label{e.dnls.1}
\Omega \mathcal{A} + \gamma_p |\mathcal{A}|^{2p} \mathcal{A} = \Delta \mathcal{A}, \quad \mathcal{A} \in \ell^2(\ZZ^d),
\end{equation}
where $\Omega$ is parameter and $\gamma_p$ is a numerical coefficient
given by (\ref{gamma-p}).  We restrict consideration to the case of
dNLS solitons given by real $\mathcal{A}$, for which we introduce the
Jacobian operator for the stationary dNLS equation (\ref{e.dnls.1}) at
$\mathcal{A}$:
\begin{equation}
\label{Jacobian-dnls}
J_{\Omega} := \Omega + (1+2p) \gamma_p |\mathcal{A}|^{2p} - \Delta.
\end{equation}
Since $\sigma(\Delta) = [-4d,0]$ in $\ell^2(\ZZ^d)$ and $\mathcal{A}
\in \ell^2(\ZZ^d)$ is expected to decay exponentially at infinity, we
need to consider $\Omega$ in $\RR \backslash [-4d,0]$.

\begin{remark}
Since the discrete solitons in the focusing stationary dNLS equation
(\ref{stat-dnls-focusing}) exist for $\tilde{\Omega} > 0$ \cite{Wei99}, the
staggering transformation (\ref{stag-transf}) suggests that the
discrete solitons in the defocusing stationary dNLS equation
(\ref{e.dnls.1}) exist for $\Omega < -4d$.
\label{rem-existence}
\end{remark}

Assuming existence of a dNLS soliton $\mathcal{A}$ in the stationary
dNLS equation (\ref{e.dnls.1}) for some $\Omega \in \RR \backslash
[-4d,0]$ and invertibility of $J_{\Omega}$ at this $\mathcal{A}$ in
(\ref{Jacobian-dnls}), we will prove existence and uniqueness of the
branch $\omega(\eps)$ and $\{ A^{(m)}(\eps) \}_{m \in \mathbb{N}} \in
\ell^{2,2}(\ZZ;\ell^2(\ZZ^d))$ of fundamental breathers satisfying the
limiting conditions:
\begin{equation}
\label{breather-limit-advanced}
\lim_{\eps \to 0} \frac{\omega(\eps) - 1}{\eps} = -\frac{1}{2} \Omega, \qquad
\lim_{\eps \to 0} A^{(m)}(\eps) = \left\{ \begin{array}{ll}
\mathcal{A}, \quad & m = 1, \\ 0, \quad & m \neq 1. \end{array} \right.
\end{equation}
The following theorem gives the existence and uniqueness result for breathers.

\begin{theorem}
\label{theorem-existence}
Fix $p \in \mathbb{N}$. Assume the existence of real $\mathcal{A} \in
\ell^2(\ZZ^d)$ in the stationary dNLS equation (\ref{e.dnls.1}) for
some $\Omega \in \RR \backslash [-4d,0]$ such that the Jacobian
operator $J_{\Omega}$ at this $\mathcal{A}$ in (\ref{Jacobian-dnls})
has trivial null space in $\ell^2(\ZZ^d)$. There exists $\eps_0 > 0$
and $C_0 > 0$ such that the breather equation (\ref{e.fou.eq}) for
every $\eps \in (0,\eps_0)$ admits a unique $C^{\omega}$
\simone{(analytic)} solution branch $\omega(\eps) \in \mathbb{R}$
and $\{ A^{(m)}(\eps) \}_{m \in \mathbb{N}} \in
\ell^{2,2}(\ZZ;\ell^2(\ZZ^d))$ satisfying the bounds
\begin{equation}
\label{error-omega}
\left| \omega(\eps) - 1 + \frac{1}{2} \eps \Omega \right| \leq C_0 \eps^2
\end{equation}
and
\begin{equation}
\label{error-A}
\| A^{(0)} \|_{\ell^2(\ZZ^d)} + \| A^{(1)} - \mathcal{A} \|_{\ell^2(\ZZ^d)} + \sum_{m \geq 2} \| A^{(m)} \|_{\ell^2(\ZZ^d)} \leq C_0 \eps,
\end{equation}
for every $\eps \in (0,\eps_0)$.
\end{theorem}

\begin{remark}
In order to explain the relevance of the stationary dNLS equation
(\ref{e.dnls.1}), we set $\omega^2 = 1 - \eps \Omega$, where $\Omega$ is fixed
independently of $\eps$, and rewrite equation
\eqref{e.fou.eq} for $m = 1$ after dividing it by $\eps$.  This
procedure yields the bifurcation equation:
\begin{displaymath}
\Omega A^{(1)}(\eps) + \frac1{2\pi}\int_{-\pi}^\pi U^{1+2p}(\tau,\eps)e^{-\im \tau} d\tau = \Delta A^{(1)}(\eps),
\end{displaymath}
where $U(\tau,\eps)$ is given by the Fourier series (\ref{FSsolution})
with amplitudes $\{A^{(m)}(\eps) \}_{m \in \mathbb{Z}}$ satisfying
symmetries (\ref{property-1}) and (\ref{property-2}).  Formally, at
the leading order (\ref{breather-limit}), we have:
\begin{eqnarray}
\label{eq-Omega}
\Omega  A^{(1)}(0) + \frac1{2\pi}\int_{-\pi}^\pi
    \left[ A^{(1)}(0) e^{\im \tau} +
  \overline{A^{(1)}(0)}e^{-\im \tau} \right]^{1+2p} e^{-\im \tau}d\tau = \Delta A^{(1)}(0)
\end{eqnarray}
By expanding
\begin{eqnarray*}
\left[ A^{(1)}(0) e^{\im \tau} +
  \overline{A^{(1)}(0)}e^{-\im \tau} \right]^{1+2p} =  \sum_{k=0}^{1+2p}
    \binom{1+2p}{k} \tond{A^{(1)}(0)}^k
\tond{\overline{A^{(1)}(0)}}^{1+2p-k} e^{\im (2k-2p-1)\tau}
\end{eqnarray*}
and evaluating the integral in (\ref{eq-Omega}),  \simone{we get the only nonzero term
for $k = p+1$. As a result, equation (\ref{eq-Omega}) becomes}  the limiting dNLS equation
(\ref{e.dnls.1}) with $\mathcal{A} := A^{(1)}(0)$. \label{remark-dnls}
\end{remark}

\begin{proof}
In order to solve the breather equation \eqref{e.fou.eq} as $\eps \to
0$ near the limiting solution (\ref{breather-limit}), we proceed with
the classical Lyapunov-Schmidt decomposition \simone{\cite{AmbPro} (see
applications of this method in a similar context in \cite{BamPP10,PelS12})}.
We introduce the Hilbert spaces
\begin{displaymath}
X_2:= H^2_{\rm per}([-\pi,\pi];\ell^2(\ZZ^d)),\qquad\qquad X_0:= L^2_{\rm per}([-\pi,\pi];\ell^2(\ZZ^d))
\end{displaymath}
and the dual spaces under the Fourier series (\ref{FSsolution}):
\begin{displaymath}
\hat{X}_2:= \ell^{2,2}(\ZZ;\ell^2(\ZZ^d)),\qquad\qquad \hat{X}_0:= \ell^2(\ZZ;\ell^2(\ZZ^d)).
\end{displaymath}
The breather solution $U$ is an element of $X_2$, which is uniquely
identified by the sequence $A$ in $\hat{X}_2$. In other words, a
solution is given by a sequence of Fourier coefficients
$\graff{A^{(m)}}_{m \in \ZZ}$ in $\ell^{2,2}(\ZZ)$, where each Fourier
coefficient $A^{(m)}$ is a complex sequence $A^{(m)} =
\graff{A^{(m)}_n}_{n\in\ZZ^d}$ in $\ell^2(\ZZ^d)$. The Sobolev norm in
space $\hat{X}_2$ is given by
\begin{displaymath}
\norm{A}_{\hat{X}_2} = \left( \sum_{m\in\ZZ}(1+|m|^2)^2\norm{A^{(m)}}_{\ell^2(\ZZ^d)}^2 \right)^{1/2}.
\end{displaymath}
Let us introduce also the linear operator $L_{\omega} : X_2
\rightarrow X_0$, which is given in Fourier space by $\hat{L}_{\omega}
: \hat{X}_2 \rightarrow \hat{X}_0$:
\begin{equation}
\label{e.l.omega}
\tond{\hat{L}_\omega A}^{(m)} = (1-m^2\omega^2)A^{(m)}, \quad m \in \ZZ.
\end{equation}

\simone{Let $\{ e_n \}_{n \in \ZZ^d}$ be a set of unit vectors in $\ell^2(\ZZ^d)$.
We define the linear subspace $V_2 = {\rm span}(\cup_{n \in \ZZ^d} e_n e^{\im \tau},
\cup_{n \in \ZZ^d} e_n e^{-\im \tau})$ as the
kernel of $L_{\omega = 1}$ in $X_2$ and the linear subspace $W_2$ as its orthogonal
complement in $X_2=V_2\oplus W_2$.} In the Fourier space, we set
$\hat{V}_2$ as a kernel of $\hat{L}_{\omega = 1}$ in $\hat{X}_2$ and
$\hat{W}_2$ as its orthogonal complement in $\hat{X}_2 = \hat{V}_2 \oplus
\hat{W}_2$.  In a similar way, we introduce the range subspace $W_0$
for the operator $L_{\omega = 1}$, which is \simone{a subspace of $X_0$ 
and is orthogonal to $V_2$, so that $X_0 = V_2 \oplus W_0$}, and similarly $\hat{X}_0 = \hat{V}_2 \oplus
\hat{W}_0$.  Any element of $\hat{X}_2$ can be decomposed into
\begin{equation}
\label{decom-A}
A = A^\sharp + A^\flat\ ,\qquad A^\sharp\in \hat{V}_2\ ,\qquad A^\flat\in
\hat{W}_2\ .
\end{equation}

The breather equation \eqref{e.fou.eq} can be written in the abstract form:
\begin{equation}
\label{e.fou.eq.2}
F(A,\omega,\eps) := \hat{L}_\omega A + \eps N(A) - \eps \Delta A = 0\ ,
\end{equation}
where $N(A)$ is the nonlinear term. If $p \in \mathbb{N}$, then the
nonlinear map $F(A,\omega,\eps) : \hat{X}_2 \times \mathbb{R} \times
\mathbb{R} \rightarrow \hat{X}_0$ is $C^{\omega}$ in its variables.
The nonlinear equation (\ref{e.fou.eq.2}) is projected onto
$\hat{V}_0$ and $\hat{W}_0$, thus yields the following two equations:
\begin{equation}
\label{kernel-range}
\Pi_{\hat{V}_0} F(A^\sharp + A^\flat,\omega,\eps)= 0\ ,\qquad\qquad \Pi_{\hat{W}_0} F(A^\sharp + A^\flat,\omega,\eps)= 0\ .
\end{equation}
The former one is known as the kernel equation and the latter one is
known as the range equation.  We shall solve the range equation for
small $\eps$ assuming that $|\omega - 1| = \mathcal{O}(\eps)$ by using
the implicit function theorem.

Exploiting the fact that $\hat{V}_0$ and $\hat{W}_0$ are invariant
under $\Delta$ and that $\hat{L}_{\omega=1} A^\sharp=0$ by definition,
the range equation in (\ref{kernel-range}) takes the form
\begin{equation}
\label{range-eq}
\tond{\hat{L}_{\omega}-\eps\Delta}A^\flat + \eps \Pi_{\hat{W}_0} N(A^\sharp+A^\flat)= 0\ .
\end{equation}
The perturbed linear operator $\hat{L}_\omega-\eps\Delta$ can be
inverted on $\hat{W}_0$ for $\eps$ small enough if $|\omega - 1| =
\mathcal{O}(\eps)$.  Indeed, first write using Neumann series
\begin{equation}
  \label{neumann-formula}
  \tond{\hat{L}_\omega-\eps\Delta}^{-1} =
  \quadr{\simone{\sum_{k=0}^{\infty}} \tond{\eps \hat{L}_\omega^{-1}\Delta}^k} \hat{L}_\omega^{-1}\ ,
\end{equation}
where $\hat{L}_\omega^{-1}$ is well defined on $\hat{W}_0$ thanks to
the diagonal form:
\begin{displaymath}
(\hat{L}_\omega^{-1} A)^{(m)} = \frac1{1-m^2\omega^2}A^{(m)}\ ,\qquad m\neq
  \pm1\ .
\end{displaymath}
Let us introduce a parametrization of $\omega$ by
\begin{equation}
\label{param-omega}
\omega^2 = 1 - \eps \Omega,
\end{equation}
where $\Omega$ is fixed independently of $\eps$. It follows by elementary
computation that there exists $\epsilon_*(\Omega)$ that only depends
on $\Omega$ such that for every $\eps \in (0,\eps_*(\Omega))$,
\begin{displaymath}
|1-m^2 \omega^2 | = |1 - m^2 (1-\eps \Omega)| > \frac12 (1+m^2)\ , \quad \forall m \in \ZZ \backslash \{-1,1\},
\end{displaymath}
thus obtaining the estimate
\begin{displaymath}
\|\hat{L}_\omega^{-1}\|_{\hat{W}_0 \rightarrow \hat{W}_2} \leq 2,
\end{displaymath}
and consequently
\begin{displaymath}
\| \eps \hat{L}_\omega^{-1} \Delta \|_{\hat{W}_0 \rightarrow \hat{W}_2} \leq 8 d \eps.
\end{displaymath}
By Neumann formula (\ref{neumann-formula}) there exists \simone{$\eps_0 > 0$ and $C_0 > 0$
such that for every $\eps \in (0,\eps_0)$,
\begin{equation}
\label{bound-on-inverse}
\|(\hat{L}_\omega - \eps \Delta)^{-1}\|_{\hat{W}_0 \rightarrow \hat{W}_2} \leq C_0.
\end{equation}
For example, we can take $\eps_0:= \min\{ \eps^*(\Omega),(8d)^{-1}\}$.}

Since $X_2$ is a Banach algebra with respect to \simone{componentwise multiplication
of sequence of functions} and
$\hat{X}_2$ is a Banach algebra with respect to convolution, the
nonlinear term $N(A)$ in (\ref{range-eq}) is closed in $\hat{X}_2$. By
writing the range equation as the fixed-point equation for $A^\flat$:
\begin{equation}
\label{fixed-point-eq}
A^\flat = -\eps \left( \hat{L}_{\omega}-\eps \Delta \right)^{-1} \Pi_{\hat{W}_0} N(A^\sharp+A^\flat)
\end{equation}
and using the implicit function theorem thanks to the parametrization
(\ref{param-omega}) and the uniform bound (\ref{bound-on-inverse}), we
conclude that for every $\eps \in (0,\eps_0)$, $\Omega \in
\mathbb{R}$, and $A^{\sharp} \in \hat{V}_2 \subset \hat{X}_2$, there
exists a unique solution $A^\flat \in \hat{W}_2 \subset \hat{X}_2$ to
the fixed-point equation (\ref{fixed-point-eq}) such that the mapping
$(A^{\sharp},\Omega,\epsilon) \rightarrow A^{\flat}$ is $C^{\omega}$
and the solution is as small as $\mathcal{O}(\eps)$ thanks to the
leading order approximation
\begin{equation}
\label{e.Aflat.exp}
A^{\flat} = -\eps \hat{L}_{\omega}^{-1}\Pi_{\hat W_0}N(A^\sharp) + \Oscr(\eps^2)\ ,
\end{equation}
which provides the bound
\begin{equation}
\label{e.Aflat.app}
\| A^\flat \|_{\hat{X}_2} \leq C \eps,
\end{equation}
for some $\eps$-independent $C$.

Inserting the parametrization (\ref{param-omega}) and the mapping
$(A^{\sharp},\Omega,\epsilon) \rightarrow A^{\flat}$ into the kernel
equation in (\ref{kernel-range}) and dividing by $\eps$, we obtain
\begin{displaymath}
\Omega A^\sharp - \Delta A^\sharp + \Pi_{\hat{V}_0} N(A^\sharp + A^{\flat}(A^{\sharp},\Omega,\epsilon)) = 0.
\end{displaymath}
Thanks to the computations in Remark \ref{remark-dnls} and the bound
\eqref{e.Aflat.app}, one can rewrite the kernel equation explicitly in
terms of the real-valued amplitude $A^{(1)}$ as follows:
\begin{equation}
f(A^{(1)},\Omega,\eps) := \Omega A^{(1)} - \Delta A^{(1)} + \gamma_p A^{(1)}|A^{(1)}|^{2p} + \eps R(A^{(1)},\Omega,\eps) = 0,
\label{e.ker.1}
\end{equation}
where $R(A^{(1)},\Omega,\eps) : \ell^2(\ZZ^d) \times \mathbb{R} \times
\mathbb{R} \to \ell^2(\ZZ^d)$ is $C^{\omega}$ and bounded as $\eps \to
0$ thanks to the bound (\ref{e.Aflat.app}). Thanks to the assumptions
of the theorem, $\mathcal{A} \in \ell^2(\ZZ^d)$ is a root of
\begin{equation}
\label{e.A.exists}
f(\mathcal{A},\Omega,0)=0
\end{equation}
and
\begin{equation}
\label{e.L+}
D_{A^{(1)}} f (\mathcal{A},\Omega,0) = J_{\Omega}
\end{equation}
is a bounded and invertible operator on $\ell^2(\ZZ^d)$.
By the implicit function theorem, there exists $\eps_1<\eps_0$ such
that for every $\eps \in (0,\eps_1)$ and $\Omega \in \mathbb{R}$ for
which $\mathcal{A} \in \ell^2(\ZZ^d)$ exists in (\ref{e.A.exists}) and
$J_{\Omega}$ is invertible in (\ref{e.L+}), there exists a unique
solution $A^{(1)} \in \ell^2(\ZZ^d)$ to the kernel equation
(\ref{e.ker.1}) such that the mapping $(\Omega,\epsilon) \rightarrow
A^{(1)}$ is $C^{\omega}$ and the solution satisfies the bound
\begin{equation}
\label{e.A.solution}
\| A^{(1)} - \mathcal{A} \|_{\ell^2(\ZZ^d)} \leq C \eps,
\end{equation}
for some $\eps$-independent $C$. Combining (\ref{e.Aflat.app}) and
(\ref{e.A.solution}) with the decompositions (\ref{decom-A}) and
(\ref{param-omega}) yields bounds (\ref{error-omega}) and
(\ref{error-A}).
\end{proof}

\section{Stability via Lyapunov-Schmidt decomposition}

Linearizing $u(t) = U(\tau) + w(t)$ of the dKG equation
(\ref{e.KG.scaled}) at the breather solution $U \in H^2_{\rm
  per}([-\pi,\pi];\ell^2(\ZZ^d))$ with $\tau = \omega t$ yields the
linearized dKG equation:
\begin{equation}
  \label{lin-KG}
\ddot{w} + w + \eps (1+2p) U^{2p} w = \eps \Delta w.
\end{equation}
By Floquet theorem, every solution of the $2\pi$-periodic linear
equation (\ref{lin-KG}) can be represented in the form $w(t) = W(\tau)
e^{\lambda t}$, where $\lambda \in \mathbb{C}$ is the spectral
parameter and $W \in H^2_{\rm per}([-\pi,\pi];\ell^2(\ZZ^d))$ is
an eigenfunction of the spectral problem:
\begin{equation}
  \label{spectral-KG}
\omega^2 W'' + 2 \lambda \omega W' + \lambda^2 W + W + \eps (1+2p) U^{2p} W = \eps \Delta W.
\end{equation}
Let us represent $W \in H^2_{\rm per}([-\pi,\pi];\ell^2(\ZZ^d))$
by the Fourier series:
\begin{equation}
  \label{FSsolution-W}
W(\tau) = \sum_{m\in\ZZ} B^{(m)} e^{\im m \tau}.
\end{equation}
With the help of (\ref{FSsolution}) and (\ref{FSsolution-W}), the
spectral problem (\ref{spectral-KG}) is rewritten in
Fourier coefficients as
\begin{equation}
\label{Fourier-KG}
\left[ 1 + (\lambda + \im m\omega)^2 \right] B^{(m)} +
\frac{\eps (1+2p)}{2\pi} \int_{-\pi}^{\pi} U^{2p}(\tau) W(\tau) e^{-\im m\tau} d\tau = \eps \Delta B^{(m)}.
\end{equation}
No symmetry reductions exist generally for the Fourier coefficients
$\{ B^{(m)} \}_{m \in \ZZ}$.

At $\eps=0$ and $\omega = 1$, the spectral problem (\ref{Fourier-KG})
admits a double set of eigenvalues $\lambda$ defined by
\begin{equation}
\label{points-spectral}
\Sigma_{\pm} := \left\{ \im (\pm 1 - m), \quad m \in \mathbb{Z} \right\},
\end{equation}
where $\Sigma_+ = \Sigma_-$ and each eigenvalue has infinite
multiplicity due to the lattice $\ZZ^d$.  In terms of the Floquet
multipliers
\begin{equation}
\label{Floquet}
\mu := e^{\lambda T} = e^{2\pi \lambda/\omega},
\end{equation}
all eigenvalues at $\epsilon = 0$ and $\omega = 1$ correspond to the
same Floquet multiplier $\mu = 1$.

\begin{remark}
The degeneracy of the Floquet multiplier $\mu$ in (\ref{Floquet}) is understood in terms of the following symmetry
for the spectral problem (\ref{Fourier-KG}). Fix $k \in \mathbb{Z}$
and apply transformation
$$
\lambda = \im k + \tilde{\lambda}, \quad m = -k + \tilde{m}, \quad B^{(m)} = \tilde{B}^{(\tilde{m})}.
$$
The eigenvalue-eigenvector pair $\left( \tilde{\lambda},\{
\tilde{B}^{(\tilde{m})}\}_{\tilde{m} \in \ZZ}\right)$ satisfies the
same spectral problem (\ref{Fourier-KG}) but in tilde
variables. Therefore, the spectral problem (\ref{Fourier-KG}) near
every nonzero point $\lambda \in \Sigma_{\pm}$ repeats its behavior
near $\lambda = 0$. It is hence sufficient to consider the
spectral problem (\ref{Fourier-KG}) near $\lambda = 0$.
\end{remark}

Let us review the spectral stability problem for the dNLS equation (\ref{dnls-introduction}).
The dNLS soliton $a(\tau) = e^{-\frac{i}{2} \Omega \tau} \mathcal{A}$ is defined by solutions
of the stationary dNLS equation (\ref{e.dnls.1}) with real $\mathcal{A} \in
\ell^2(\ZZ^d)$. Linearizing with the expansion $a(\tau) = e^{-\frac{i}{2} \Omega \tau}
\left[ \mathcal{A} + b(\tau) \right]$ yields the linearized dNLS equation:
\begin{equation}
\label{lin-dnls}
2i b' + \left( \Omega - \Delta + \gamma_p (p+1) \mathcal{A}^{2p} \right) b + \gamma_p p \mathcal{A}^{2p} \bar{b} = 0.
\end{equation}
Separating variables by
\simone{
$$
b(\tau) = (b_+ + \im b_-) e^{\Lambda \tau} + (\bar{b}_+ + \im \bar{b}_-) e^{\bar{\Lambda} \tau}
$$
and
$$
\bar{b}(\tau) = (b_+ - \im b_-) e^{\Lambda \tau} + (\bar{b}_+ - \im \bar{b}_-) e^{\bar{\Lambda} \tau},
$$
where $\Lambda \in \mathbb{C}$ is the spectral parameter and $(b_+,b_-) \in
\ell^2(\ZZ^d) \times \ell^2(\ZZ^d)$ is a complex-valued eigenfunction}, yields the
spectral problem:
\begin{equation}
  \label{spectral-NLS}
\left[ \begin{array}{cc} 0 & -(\Omega - \Delta + \gamma_p \mathcal{A}^{2p}) \\
\Omega - \Delta + \gamma_p (1+2p) \mathcal{A}^{2p} & 0 \end{array} \right] \left[ \begin{array}{c} b_+ \\ b_- \end{array} \right] =
2 \Lambda \left[ \begin{array}{c} b_+ \\ b_- \end{array} \right].
\end{equation}
The spectral problem (\ref{spectral-NLS}) can be written in
the Hamiltonian form $\mathcal{J} \mathcal{H}''(\mathcal{A}) \vec{f} =
2 \Lambda \vec{f}$, where $\vec{f} = (b_+,b_-)^T$,
$$
\mathcal{J} = \left[ \begin{array}{cc} 0 & -1 \\ 1 & 0 \end{array} \right], \quad
\mathcal{H}''(\mathcal{A}) = \left[ \begin{array}{cc} \Omega - \Delta + \gamma_p (1+2p) \mathcal{A}^{2p} & 0 \\
0 & \Omega - \Delta + \gamma_p \mathcal{A}^{2p} \end{array} \right].
$$
The first diagonal entry in $\mathcal{H}''(\mathcal{A})$ coincides
with the Jacobian operator (\ref{Jacobian-dnls}) for the stationary
dNLS equation (\ref{e.dnls.1}).

\begin{remark}
Since $\mathcal{H}''(\mathcal{A})$ and $\Omega - \Delta$ are bounded
operators in $\ell^2(\ZZ^d)$, whereas $\Omega \in \RR \backslash
[-4d,0]$ and $\mathcal{A}^{2p}$ decays exponentially at infinity, the
operator $(\Omega - \Delta)^{-1} \mathcal{A}^{2p}$ is a compact
(Hilbert--Schmidt) operator.  As a result,
$\sigma_c(\mathcal{H}''(\mathcal{A})) =
[\Omega,\Omega + 4d]$ and $\sigma_d(\mathcal{H}''(\mathcal{A}))$
consists of finitely many eigenvalues of finite multiplicities, where
$\sigma_c$ and $\sigma_d$ denotes the absolutely continuous and
discrete spectra of the self-adjoint operator
$\mathcal{H}''(\mathcal{A})$ in the Hilbert space $\ell^2(\ZZ^d)$.
\label{rem-stability}
\end{remark}

It follows from Remark \ref{rem-stability} that if $\Omega < -4d$ (see
Remark \ref{rem-existence}), there exist finitely many positive
eigenvalues of $\sigma_d(\mathcal{H}''(\mathcal{A}))$, whereas if
$\Omega > 0$, there exist finitely many negative eigenvalues of
$\sigma_d(\mathcal{H}''(\mathcal{A}))$.  In either case, the stability
theory in linear Hamiltonian systems \cite{KapProm2013,Pelin2011} is
applied to conclude that there exist finitely many eigenvalues
$\Lambda$ with ${\rm Re}(\Lambda) \neq 0$ in the spectral problem
(\ref{spectral-NLS}). The continuous spectrum of $\mathcal{J}
\mathcal{H}''(\mathcal{A})$ coincides with the purely continuous spectrum
of $\mathcal{J} \mathcal{H}''(0)$ and is located on
\begin{equation}
\label{cont-spectrum}
\sigma_c(\mathcal{J} \mathcal{H}''(\mathcal{A})) = \{ i [\Omega, \Omega + 4 d] \} \cup \{ -i [\Omega, \Omega + 4 d] \}.
\end{equation}
The following theorem guarantees the persistence of simple isolated
eigenvalues of the spectral problem (\ref{spectral-NLS}) in the spectral
problem (\ref{Fourier-KG}) near $\lambda = 0$.

\begin{theorem}
\label{theorem-linear}
Under the assumption of Theorem \ref{theorem-existence}, assume that
$\Lambda \in \mathbb{C}$ is a simple isolated eigenvalue of the
spectral problem (\ref{spectral-NLS}) such that $2 \Lambda \notin
\sigma_c(\mathcal{J} \mathcal{H}''(\mathcal{A}))$ and
$(b_+,b_-) \in \ell^2(\ZZ^d) \times \ell^2(\ZZ^d)$.
There exists $\eps_0 > 0$ and $C_0 > 0$ such that the spectral problem
(\ref{Fourier-KG}) for every $\eps \in (0,\eps_0)$ admits a unique
$C^{\omega}$ branch of the eigenvalue--eigenvector pair with
$\lambda(\eps) \in \CC$ and $\{ B^{(m)}(\eps) \}_{m \in \mathbb{N}}
\in \ell^{2,2}(\ZZ;\ell^2(\ZZ^d))$ satisfying
\begin{equation}
\label{error-lambda}
\left| \lambda(\eps) - \eps \Lambda \right| \leq C_0 \eps^2,
\end{equation}
\begin{equation}
\label{error-B}
\| B^{(1)} - b_+ - \im b_-\|_{\ell^2(\ZZ^d)}
+ \| B^{(-1)} - b_+ + \im b_-\|_{\ell^2(\ZZ^d)}  \leq C_0 \eps,
\end{equation}
and
\begin{equation}
\label{error-B-rest}
\| B^{(0)} \|_{\ell^2(\ZZ^d)} + \sum_{m \geq 2} \| B^{(m)} \|_{\ell^2(\ZZ^d)} \leq C_0 \eps,
\end{equation}
for every $\eps \in (0,\eps_0)$.
\end{theorem}

\begin{proof}
We adopt the same Hilbert spaces as those used in the proof of Theorem
\ref{theorem-existence}.  Any element of $\hat{X}_2$ can be decomposed
into
\begin{equation}
\label{decom-B}
B = B^\sharp + B^\flat\ ,\qquad B^\sharp\in \hat{V}_2\ ,\qquad B^\flat\in
\hat{W}_2.
\end{equation}
We assume that $\omega(\eps)$ and $\{A^{(m)}(\eps)\}_{m \in
  \mathbb{Z}}$ are given by Theorem \ref{theorem-existence} with the
error bounds (\ref{error-omega}) and (\ref{error-A}). Let us introduce
the linear operator $\hat{M}_{\lambda,\omega} : \hat{X}_2 \rightarrow
\hat{X}_0$:
\begin{equation}
\label{M-oper}
\tond{\hat{M}_{\lambda,\omega} B}^{(m)} = \left[ 1+ (\lambda + \im m \omega)^2 \right] B^{(m)}, \quad m \in \ZZ.
\end{equation}
The spectral problem (\ref{Fourier-KG}) for Fourier coefficients can
be written in the abstract form:
\begin{equation}
\label{F-oper}
F(B,\lambda,\eps) := \hat{M}_{\lambda,\omega(\eps)} B + \eps S(A(\eps),B) - \eps \Delta B = 0,
\end{equation}
where $S(A(\eps),B)$ is the linear map on $B$ obtained from the
nonlinear term $N(A)$.  Since $p \in \mathbb{N}$, the map
$F(B,\lambda,\eps) : \hat{X}_2 \times \mathbb{C} \times \mathbb{R}
\rightarrow \hat{X}_0$ is $C^{\omega}$ in its arguments. Projecting
equation (\ref{F-oper}) onto $\hat{V}_0$ and $\hat{W}_0$ yields the
following range and kernel equations:
\begin{equation}
\label{kernel-range-lin}
\Pi_{\hat{V}_0} F(B^\sharp + B^\flat,\lambda,\eps)= 0,\qquad\qquad \Pi_{\hat{W}_0} F(B^\sharp + B^\flat,\lambda,\eps)= 0\ .
\end{equation}

The range equation in system (\ref{kernel-range-lin}) can be solved in
the same way as the range equation in system (\ref{kernel-range}). By
using the implicit function theorem, for every $\eps \in (0,\eps_0)$,
$\Lambda \in \mathbb{C}$, and $B^{\sharp} \in \hat{V}_2 \subset
\hat{X}_2$, there exists a unique solution $B^\flat \in \hat{W}_2
\subset \hat{X}_2$ of the range equation $\Pi_{\hat{V}_0} F(B^\sharp +
B^\flat,\eps \Lambda,\eps)= 0$ such that the mapping
$(B^{\sharp},\Lambda,\epsilon) \rightarrow B^{\flat}$ is $C^{\omega}$
and the solution is as small as $\mathcal{O}(\eps)$ thanks to the
bound
\begin{equation}
\label{e.Aflat.app-lin}
\| B^\flat \|_{\hat{X}_2} \leq C \eps,
\end{equation}
for some $\eps$-independent $C$.

Inserting $\omega = 1 - \frac{1}{2} \epsilon \Omega +
\mathcal{O}(\epsilon^2)$, $\lambda = \epsilon \Lambda$, and the
$C^{\omega}$ mapping $(B^{\sharp},\Lambda,\epsilon) \rightarrow
B^{\flat}$ into the kernel equation in system (\ref{kernel-range-lin})
and dividing by $\eps$, we obtain the following system of two
equations on the two amplitudes $(B^{(1)},B^{(-1)})$:
\begin{eqnarray}
\nonumber
(\Omega \pm 2 i \Lambda) B^{(\pm 1)}
- \Delta B^{(\pm 1)} + \gamma_p \mathcal{A}^{2p} \left[ (p+1) B^{(\pm 1)} + p B^{(-1)} \right] \\
+ \eps R^{(\pm 1)}(B^{(1)},B^{(-1)},\Lambda,\eps) = 0,
\label{ampl-B}
\end{eqnarray}
where $R^{(\pm 1)}(B^{(1)},B^{(-1)},\Lambda,\eps) : \ell^2(\ZZ^d)
\times \ell^2(\ZZ^d) \times \mathbb{R} \times \mathbb{R} \to
\ell^2(\ZZ^d)$ is a linear map on $(B^{(1)},B^{(-1)})$ with
$C^{\omega}$ coefficients which are bounded as $\eps \to 0$ thanks to
the bound (\ref{e.Aflat.app-lin}).  In the derivation of numerical
coefficients in (\ref{ampl-B}), we have used the following explicit
computation:
\begin{eqnarray*}
& \phantom{t} & \frac1{2\pi}\int_{-\pi}^\pi
    \left[ \mathcal{A} e^{\im \tau} +
\mathcal{A} e^{-\im \tau} \right]^{2p} \left[ B^{(1)} e^{\im \tau} + B^{(-1)} e^{-\im \tau} \right] e^{\mp \im \tau}d\tau  \\
& = & \sum_{k=0}^{2p} \binom{2p}{k} \mathcal{A}^{2p} \frac1{2\pi} \int_{-\pi}^\pi
\left[ B^{(1)} e^{\im (2k-2p+1 \mp 1)\tau} + B^{(-1)} e^{\im (2k-2p-1 \mp 1) \tau}\right] d\tau\\
& = & \frac{p+1}{2p+1} \gamma_p \mathcal{A}^{2p} B^{(\pm 1)} + \frac{p}{2p+1} \gamma_p \mathcal{A}^{2p} B^{(-1)}.
\end{eqnarray*}
At $\eps = 0$, the system (\ref{ampl-B}) becomes the spectral problem
(\ref{spectral-NLS}) in variables $B^{(\pm 1)} = b_+ \pm \im
b_-$. It is assume that $\Lambda$ is a simple isolated eigenvalue
in the spectral problem (\ref{spectral-NLS}) with $2 \Lambda \notin
\sigma_c(\mathcal{J} \mathcal{H}''(\mathcal{A}))$ and a related
eigenvector $(b_+,b_-) \in \ell^2(\ZZ^d) \times
\ell^2(\ZZ^d)$.  For $\eps \neq 0$, the eigenvalue $\Lambda$ becomes
the characteristic root of the linear system (\ref{ampl-B}). By the
analytic perturbation theory for closed linear operators (see Theorem
1.7 in Chapter VII on p. 368 in \cite{Kato}), simple characteristic
roots and the associated eigenvectors are continued in $\eps$ as
$C^{\omega}$ functions. This completes justification of the bounds
(\ref{error-lambda}) and (\ref{error-B}).
\end{proof}

\begin{remark}
\label{rem-Krein}
If $2 \Lambda \in i \mathbb{R} \backslash \sigma_c(\mathcal{J}
\mathcal{H}''(\mathcal{A}))$, the bound (\ref{error-lambda}) is not
sufficient to guarantee that the eigenvalue $\lambda$ remains on $\im
\mathbb{R}$.
\end{remark}

In order to obtain a definite prediction that the simple isolated
eigenvalue $\Lambda \in \im \mathbb{R}$ of the spectral problem
(\ref{spectral-NLS}) persist as a simple isolated eigenvalue $\lambda
\in \im \mathbb{R}$ of the spectral problem (\ref{Fourier-KG}), we use
the Krein signature theory for linearized Hamiltonian
systems. Consider the linearized dKG equation (\ref{lin-KG}) and define
\begin{equation}
\label{Krein}
k(w) := \im \sum_{n \in \ZZ^d} w_n \dot{\bar{w}}_n - \bar{w}_n \dot{w}_n.
\end{equation}
It is straightforward to verify that $k(w)$ is independent of $t$.
Let us represent the eigenvalue-eigenvector pair by
$w(t) = W(\tau) e^{\lambda t}$ with $\lambda \in \mathbb{C}$
and $W \in H^2_{\rm per}([-\pi,\pi];\ell^2(\ZZ^d))$. Then,
$k(w) = K(W,\lambda) e^{(\lambda + \bar{\lambda}) t}$ with
\begin{equation}
\label{Krein-eigenvalue}
K(W,\lambda) := \im \omega \sum_{n \in \ZZ^d} \left( W_n \bar{W}'_n - \bar{W}_n W'_n \right)
- \im (\lambda - \bar{\lambda}) \sum_{n \in \ZZ^d} |W_n|^2.
\end{equation}
The following lemma reproduces the main result of the Krein theory.

\begin{lemma}
\label{lem-Krein}
Let $\lambda \in \mathbb{C}$ be a simple isolated eigenvalue
in the spectral problem (\ref{spectral-KG}) with the
eigenvector $W \in H^2_{\rm per}([-\pi,\pi];\ell^2(\ZZ^d))$.
Then, $K(W,\lambda) = 0$ if ${\rm Re}(\lambda) \neq 0$ and
$K(W,\lambda) \neq 0$ if $\lambda \in \im \mathbb{R} \backslash
\{0\}$.
\end{lemma}

\begin{proof}
\simone{Let $Q := \lambda W + \omega W'$, where $W \in H^2_{\rm per}([-\pi,\pi];\ell^2(\ZZ^d))$
is an eigenvector of the spectral problem (\ref{spectral-KG}). Then, system (\ref{spectral-KG})
can be formulated in the Hamiltonian form $J H''(U) \vec{f} = \lambda \vec{f}$, where $\vec{f}
= (W,Q)$, $J^* = -J = J^{-1}$, and $H''(U)$ is self-adjoint in
$L^2_{\rm per}([-\pi,\pi];\ell^2(\ZZ^d))$. By using (\ref{Krein-eigenvalue}), we obtain that}
$$
\lambda K(W,\lambda) = i \langle H''(U) \vec{f}, \vec{f} \rangle = i \langle \vec{f}, H''(U) \vec{f} \rangle
= - \bar{\lambda} K(W,\lambda),
$$
so that if ${\rm Re}(\lambda) \neq 0$ then $K(W,\lambda) = 0$. If
$\lambda \in \im \mathbb{R} \backslash \{0\}$ is a simple isolated
eigenvalue, then we claim that $K(W,\lambda) \neq 0$.  Indeed, if we
assume $K(W,\lambda) = 0$, then there exists a generalized eigenvector
from solution of the nonhomogeneous equation
$$
J H''(U) \vec{g} = \lambda \vec{g} + \vec{f},
$$
since the condition of the Fredholm alternative theorem is satisfied:
$$
\langle J^{-1} \vec{f}, \vec{f} \rangle = \lambda^{-1} \langle H''(U) \vec{f}, \vec{f} \rangle
= -i K(W,\lambda) = 0.
$$
Therefore, $\lambda$ is at least a double eigenvalue in contradiction
with the assumption that $\lambda$ is simple. Therefore, $K(W,\lambda)
\neq 0$.
\end{proof}

Equipped with Lemma \ref{lem-Krein}, we can now prove an analogue of
Theorem \ref{theorem-linear} about persistence of simple isolated
eigenvalues on $\im \mathbb{R}$.

\begin{theorem}
\label{theorem-Krein}
Under the assumption of Theorem \ref{theorem-existence}, assume that
$\Lambda \in \im \mathbb{R} \backslash \{0\}$ is a simple isolated
eigenvalue of the spectral problem (\ref{spectral-NLS})
\simone{such that $2 \Lambda \notin
\sigma_c(\mathcal{J} \mathcal{H}''(\mathcal{A}))$ and }
$(b_+,b_-) \in \ell^2(\ZZ^d) \times \ell^2(\ZZ^d)$.  There
exists $\eps_0 > 0$ and $C_0 > 0$ such that the spectral problem
(\ref{Fourier-KG}) for every $\eps \in (0,\eps_0)$ admits a unique
$C^{\omega}$ branch of the eigenvalue--eigenvector pair with
$\lambda(\eps) \in \im \RR$ and $\{ B^{(m)}(\eps) \}_{m \in
  \mathbb{N}} \in \ell^{2,2}(\ZZ;\ell^2(\ZZ^d))$ satisfying
(\ref{error-lambda}), (\ref{error-B}), and (\ref{error-B-rest}).
\end{theorem}

\begin{proof}
By Remark \ref{rem-Krein}, we only need to prove that $\lambda(\eps) =
\eps \Lambda + \mathcal{O}(\eps^2)$ remains on $\im \RR$. By
smoothness of the branch of eigenvalue-eigenvectors in $\eps$, we can
compute the limit $\eps \to 0$ for the Krein quantity $K(W,\lambda)$
in (\ref{Krein-eigenvalue}).  We obtain
\begin{equation*}
  \lim_{\eps \to 0} K(W,\lambda)  =
  2 \| B^{(1)} \|^2_{\ell^2(\ZZ^d)} - 2 \| B^{(-1)} \|^2_{\ell^2(\ZZ^d)}
  =  4 \im \langle b_-, b_+ \rangle_{\ell^2(\ZZ^d)} -
  4 \im \langle b_+, b_- \rangle_{\ell^2(\ZZ^d)},
\end{equation*}
which is the Krein quantity for the spectral problem
(\ref{spectral-NLS}). Since $\Lambda \in \im \mathbb{R} \backslash
\{0\}$ is simple and isolated, the Krein quantity for the spectral
problem (\ref{spectral-NLS}) enjoys the same properties as in Lemma
\ref{lem-Krein}. In particular, it is real and nonzero. By continuity
in $\eps$, $K(W,\lambda)$ is nonzero for every $\eps \in (0,\eps_0)$,
so that by Lemma \ref{lem-Krein}, the eigenvalue $\lambda(\eps) = \eps
\Lambda + \mathcal{O}(\eps^2)$ of the spectral problem
(\ref{spectral-KG}) satisfies ${\rm Re}(\lambda) = 0$.
\end{proof}

\begin{remark}
\simone{Assume that $\sigma_d(\mathcal{J} \mathcal{H}''(\mathcal{A}))$ contains
no eigenvalues $2 \Lambda$ with ${\rm Re}(\Lambda) \neq 0$, hence
the dNLS solitons are spectrally stable.  Theorem \ref{theorem-Krein}
implies } that the spectral stability of dNLS solitons
is transferred to the spectral stability of dKG breathers
if bifurcations of new isolated eigenvalues from the continuous spectrum
in (\ref{cont-spectrum}) do not result in the appearance of
new eigenvalues with ${\rm Re}(\lambda) \neq 0$ in the spectral
problem (\ref{spectral-NLS}). Such arguments generally follow
from the Krein theory \cite{Pelin2011}. In the anti-continuum limit
of the dNLS equation (\ref{dnls-introduction}),
one can find precise conditions excluding bifurcations of new isolated eigenvalues
from the continuous spectrum of the spectral problem (\ref{spectral-NLS}) \cite{PelS11}.
\end{remark}

\section{Long-time nonlinear stability via resonant normal forms}
\label{s:RNF}



The resonant normal form we consider here is based on the scheme
already illustrated in \cite{BamG93,BGG89}, which is suitable for
infinite dimensional Hamiltonian systems and can be implemented by
working at the level of either the Hamiltonian fields (as we decide to
do, following \cite{BamG93}) or the Hamiltonian function (as in
\cite{BGG89}).

In what follows we first present a result according to which the
Hamiltonian of our problem can be put into a resonant normal form up
to an exponentially small remainder. The truncated normal form
represents a generalized dNLS equation in the same spirit as
in~\cite{PalP16a}. We then give a theorem about the existence of a
breather for the dKG equation, exponentially close to discrete soliton
of the normal form; we stress here that such an estimate is a
significant improvement with respect to the one obtained
in~\cite{PalP16} where the two objects were proven to be only order
one close in the small parameter. As a last step, under additional
hypothesis that the dNLS soliton is a minimizer in the variational
problem (\ref{constrained-var-pr}), we state a stability result for
the discrete breathers on an exponentially long time scale. The proofs
of the above mentioned results are illustrated respectively in
Subsections~\ref{ss:p1},~\ref{ss:p2} and~\ref{ss:p3}.  \simone{Stability
  over longer, even infinite, time scales cannot be excluded a priori,
  but may require different ingredients than the present normal form
  construction. Indeed, although a standard control of the
  remainder does not generally prevent the orbit to leave a small neighbourhood
  of the breather beyond the time scale here considered, the stability could 
  in principle persist even if we are not able to ensure it.}


\subsection{Setting, preliminaries and normal form result}
\label{ss:setting}

We consider the Hamiltonian corresponding to the scaled model
\eqref{e.KG.scaled}
\begin{equation}
  \label{e.Ham.scaled}
  H = \frac12 \sum_{j\in\ZZ^d} \tond{u_j^2+v_j^2} + \frac\eps{2p+2} \sum_{j\in\ZZ^d} u_j^{2p+2}+
  \frac\eps2 \sum_{j\in\ZZ^d} \sum_{|j-h|=1}\tond{u_j-u_h}^2,
\end{equation}
where $v_j = \dot{u}_j$. The Hamiltonian \eqref{e.Ham.scaled} can be
obtained scaling both the variables $(u_n,\dot u_n)$ according to
\eqref{e.scale.orig}, and the original Hamiltonian original energy
\eqref{Ham} by $\eps^{-\frac1p}$. In the following,
\eqref{e.Ham.scaled} will be considered as a nearly integrable
Hamiltonian system
\begin{equation}
  \label{e.NIHS}
  H = G+F\ ,
  \qquad\quad
  G := \frac12 \sum_{j\in\ZZ^d}\tond{u_j^2+v_j^2}\ ,
  \qquad
  F := H-G = \Oscr(\eps)\ ,
\end{equation}
where $G$ is an integrable Hamiltonian and $F$ is a perturbation of
order $\Oscr(\eps)$.

We need some notations (we refer to Section 5 of \cite{BamG93} for
further details). We consider $z:=(u,v)$ in the complexified phase
space $\Pscr=\ell^2(\CC)\times\ell^2(\CC)$ with the usual $\ell^2$
norm, which makes it Hilbert with the usual inner product. Given
$0<R<1$ and $0<\dfr\leq\frac14$, we restrict to a ball around the
origin $B_{R,\dfr}:=\graff{z\in\Pscr\quad s.t.\quad \norm{z}<
  R(1-\dfr)}$. To deal with complex valued functions $g$ and
Hamiltonian vector fields $X_g$ on such a generic ball, we make use of
the supremum norm
\begin{equation}
  \label{e.norms}
N_\dfr(g):=\sup_{z\in B_{R,\dfr}}|g(z)| \ ,\qquad N^{\nabla}_\dfr(g):=\frac1R
\sup_{z\in B_{R,\dfr}}\norm{X_g(z)}\ .
\end{equation}
Our aim is to construct a normal form $K$ admitting a second conserved
quantity $G$
\begin{displaymath}
H=K+\Pscr\ ,\qquad\qquad \Poi{K}{G}=0\ ;
\end{displaymath}
this additional conserved quantity, which correspond to the $\ell^2$
norm, corresponds to the invariance under the rotation symmetry, given
by the periodic flow $\Phi^t_G$ of the Hamiltonian field $X_G$. The
normal form $K$ is thus a generalized dNLS model (see also
\cite{PalP16a}); given the smallness of $\Pscr$, $G$ turns out to be
an approximated conserved quantity for $H$, whose variation can be
kept bounded on exponentially long times.
\begin{theorem}
  \label{t.main}
  For any positive $\dfr\leq 1/4$, any dimension $d\geq 1$ and any $R
  < 1$, there exists $\eps^*(\dfr,d,R)$ such that, for $\eps<\eps^*$
  there exists a canonical change of coordinates $T_\Chi$ mapping
  \begin{equation}
    \label{e.incl.balls}
    B_{R,2\dfr} \subset\  T_\Chi\tond{B_{R,\dfr}}  \subset B_{R,0}
    \qquad\qquad\qquad
    B_{R,3\dfr} \subset\  T_\Chi\tond{B_{R,2\dfr}} \subset B_{R,\dfr}
  \end{equation}
  which puts the Hamiltonian \eqref{e.Ham.scaled} into the resonant
  normal form
  \begin{equation}
    \label{e.main.est.1}
    H=G+Z+\Pscr\ ,
    \qquad\qquad
    \Poi{G}{Z}=0\ ,
    \qquad\qquad
    \nvec{\Pscr}{\dfr}\leq \mu\exp\tond{-\frac1{\mu}}\ ,
  \end{equation}
  where $\mu:=\frac{12e\pi\eps}{\dfr}=\Oscr(\eps)$.  Moreover, for any
  initial datum $z_0\in B_{R,3\dfr}$, there exists a positive constant
  $C$ such that the variations of $G$ and $Z$ are bounded as follows
  \begin{equation}
    \label{e.main.est.2}
    \begin{aligned}
      |G(z(t))-G(z_0)| &< C\mu\nfun{G}{0}\ ,\\
      |Z(z(t))-Z(z_0)| &< C\mu\nfun{F}{0}\ ,
    \end{aligned}
    \qquad\qquad
    |t|\leq T^*:=\exp\tond{\frac1{\mu}}\ .
  \end{equation}
\end{theorem}

The construction is based on the linear operator $T_\Chi$ associated
to a generating sequence $\graff{\Chi_s}_{s=1}^r$, where
$\Chi_s=\Oscr(\eps^s)$, which acts recursively on $G$ and $F$ as follows
\begin{equation}
  \label{e.F.rec}
    \begin{aligned}
      T_\Chi G &= \sum_{r\geq 0} G_r\ ,
      &\qquad G_0&:=G\ ,
      &\quad G_r&:= \sum_{l=1}^r\frac{l}{r}\Poi{\Chi_l}{G_{r-l}}\ ,
      \\
      T_\Chi F &= \sum_{r\geq 0} F_r\ ,
      &\qquad F_0&:=0\ ,\quad F_1:=F\ ,
      &\quad F_r&:=\sum_{l=1}^{r-1}\frac{l}{r-1}\Poi{\Chi_l}{F_{r-l}}\ .
    \end{aligned}
\end{equation}
In the above recursive definition, it coherently turns out that
$F_r=\Oscr(\eps^r)$. Such a linear operator also provides the
close-to-the-identity nonlinear transformation
\begin{equation}
  \label{e.change.coord}
  T_\Chi z = z + \sum_{r\geq 1}z_r\ ,
  \qquad\qquad
  z_r = \sum_{l=1}^r\frac{l}{r}\Poi{\Chi_l}z_{r-l}\ .
\end{equation}
The generating sequence $\Chi$, and the corresponding transformation
$T_\Chi$, will be determined in order to put the Hamiltonian in
resonant normal form up to order $\Oscr(\eps^r)$
\begin{equation}
  \label{e.NF}
  H^{(r)} = T_\Chi H = G + Z + \Rscr^{(r+1)}\ ,
  \qquad\quad
  \Poi{G}{Z}=0\ ,
  \qquad
  \Rscr^{(r+1)} = \Oscr(\eps^{r+1})\ .
\end{equation}
Thus $\Chi=\graff{\Chi_s}$ and the normal form terms $Z=\sum_{s=1}^r
Z_s$ have to satisfy
\begin{equation}
  \label{e.om.eq}
  \Poi{G}{\Chi_s}+Z_s = \Psi_s  \ ,\qquad 1\leq s\leq r\ ,
\end{equation}
where $\Chi_s,\,Z_s$ and $\Psi_s$ are all homogeneous terms of order
$\eps^s$, with
\begin{equation}
  \label{e.Psi.rec}
  \Psi_1=F_1=F\ ,\qquad\qquad \Psi_s := \frac1s F_s +
  \sum_{l=1}^{s-1}\frac{l}{s}\Poi{\Chi_l}{Z_{s-l}}\ .
\end{equation}

At first order $r=1$, we obtain again equation \eqref{e.dnls.1} as
leading order approximation of the dKG breather. Indeed we have to put
into normal form the initial perturbation $\Psi_1:=F_1$. The first
normal form term $Z_1$ represents its average, and it turns out that
at first order the Hamiltonian $K^{(1)}$ can be given by the
corresponding dNLS model
\begin{equation}
  \label{e.K1}
    K^{(1)} = \sum_{j}|\psi_j|^2 +
    \frac{\eps}{p+1}\sum_j|\psi_j|^{2p+2} +
    {\eps}\sum_{|j-h|=1}|\psi_j-\psi_h|^2\ ,
\end{equation}
once complex coordinates are introduced
\begin{equation}
  \label{e.psi.coord}
  u_j=\psi_j+\im\overline{\psi_j}= \frac1{\sqrt2}(\zeta_j+\im\eta_j)\ ,
  \quad\Rightarrow\quad
  \psi_j=\zeta_j/\sqrt2 \ ,
  \qquad
  \im\eta_j=\overline{\zeta_j} \ ,
\end{equation}
so that the quadratic part of $K^{(1)}$ reads
$
\sum_j|\zeta_j|^2 + \eps\sum_{|j-h|=1}|\zeta_j-\zeta_h|^2\ .
$
To average the nonlinearity one follows the same calculations already
used in the Remark~\ref{remark-dnls}
\begin{displaymath}
  \frac1{2\pi}\int_0^{2\pi} u_j^{2p+2}\circ\Phi_g^t dt =
  \frac1{2\pi}\int_0^{2\pi} \frac1{2^{p+1}}\tond{\zeta_j e^{\im t} +
    \im{\eta_j} e^{-\im t}}^{2p+2} dt = \Gamma_p|\zeta_j|^{2p+2}
  \ ,
\end{displaymath}
with $\Gamma_p := \frac1{2^p}\gamma_p$; thus that the nonlinear term reads
$
\frac{\Gamma_p}{2(p+1)}\sum_j |\zeta_j|^{2p+2}\ ,
$
and its standard shape is recovered introducing the complex variable
\eqref{e.psi.coord} which allows to rescale the prefactor $2^{-p}$.
Discrete solitons of \eqref{e.K1} with frequency close to one
\begin{equation}
\label{e.ds.ansatz}
\simone{\psi = \mathcal{A}e^{\im(1-\frac\eps2\Omega)t}}
\end{equation}
are then extremizer of $\Zscr_1:=\eps^{-1}Z_1$ constrained to constant
values of the norm $G=\nu$, thus providing again \eqref{e.dnls.1} with
$\Omega=\Omega(\nu)$.

\subsection{High order approximation and nonlinear stability results}

Let us consider $K:=H-\Pscr$ in \eqref{e.main.est.1} and its equations
\begin{equation}
\label{e.nf.eq}
\dot z = X_K(z)\ ,\qquad\qquad K = G + Z\ ,\qquad\qquad \Poi{Z}{G}=0\ .
\end{equation}
To generalize the discrete soliton approximation, we rewrite the
ansatz \eqref{e.ds.ansatz} as
\begin{equation}
\label{e.ds.ansatz.gen}
\zeta_{\textrm{ds}} = \mathfrak{A} e^{\im(1-\frac12\eps\Omega)t}
\end{equation}
where $\Afr$ is the real amplitude of the soliton\footnote{notice the
  use of the gothic font instead of the calligraphic one to
  distinguish between the objects of the generalized dNLS -- given by
  the higher order normal form -- to those of the standard dNLS},
which is assumed to be small enough to belong to the domain of
validity of the normal form \eqref{e.main.est.1}; once inserted in
\eqref{e.nf.eq}, it provides the equation for $\mathfrak{A}$
\begin{equation}
\label{e.f}
f:=f_0 + \eps f_1=0\ ,\qquad
  \begin{cases}
    f_0&:= \Omega\mathfrak{A} +
    \gamma_p|\mathfrak{A}|^{2p}\mathfrak{A} - \Delta\mathfrak{A}\ ,
    \\
    f_1&:= \eps^{-2}X_{Z}(\mathfrak{A})\ ;
  \end{cases}
\end{equation}
where $f_0$ gives the standard dNLS equation \eqref{e.dnls.1}, while
$f_1$ is the perturbation due to the normal form steps $r\geq 2$; we
recall that, due to $\Poi{G}{Z}=0$, $X_{Z}$ is equivariant under the
action of the symmetry $e^{\im\theta}$
\begin{displaymath}
X_{Z}\tond{\Afr e^{\im(1-\frac12\eps\Omega)t}} =
X_{Z}\tond{\Afr} e^{\im(1-\frac12\eps\Omega)t}\ .
\end{displaymath}
The next statement represent the higher order version of
Theorem~\ref{theorem-existence}, under the same assumption on $\A$ and
$J_\Omega$: it claims the existence of the breather for the
Klein-Gordon close to the discrete soliton of the normal form $K$.

\begin{theorem}
  \label{t.hoa}
  Let $\A$ be a solution of \eqref{e.dnls.1} with $J_\Omega$ of
  \eqref{Jacobian-dnls} invertible in $\ell^2(\ZZ^d,\RR)$. Then:
  \begin{enumerate}
  \item there exists $\eps^*_1<\eps^*$ such that for any
    $0<\eps<\eps^*_1$ there exists a unique solution
    $\Afr(\Omega,\eps)$ of \eqref{e.f}, analytic in
    $\eps$. Moreover, the following estimates hold true
    \begin{equation}
      \label{e.ds.per.bound}
      \norm{\Afr-\mathcal{A}}_{\ell^2} \leq C\eps\ ,
      \qquad\qquad
      \sup\frac{\norm{(J_{\Omega,\eps}-J_\Omega)(z)}_{\ell^2(\ZZ^d;\RR)}}
               {\norm{z}_{\ell^2(\ZZ^d;\RR)}} \leq C\eps\ ,
    \end{equation}
    where $J_{\Omega,\eps}:= D_{\Afr}f(\Afr(\Omega,\eps),\Omega,\eps)$
    is the differential of $f$ evaluated at $\Afr(\Omega,\eps)$.

  \item Let
    \begin{equation}
      \label{e.br.fourier}
      \zeta_{\textrm{br}}(\tau)=\sum_m \Afr^{(m)} e^{\im
        m\tau}\ ,\qquad\qquad\tau:=\omega t\ ,
    \end{equation}
    be the Fourier expansion of the breather of $\dot z=X_H(z)$.
    Then, there exists positive $\eps^*_2<\eps^*_1$ such that for
    every $0<\eps<\eps^*_2$ the breather \eqref{e.br.fourier} admits a
    unique analytic solution branch $\omega(\eps)$ and
    $\graff{\Afr^{(m)}(\eps)}\in\ell^{2,2}(\ZZ;\ell^2(\ZZ^d))$
    satisfying the bounds
    \begin{align}
      \label{e.err-omega.2}
      |\omega(\eps)-1+\frac\eps2\Omega| &\leq C\eps^2\ ,
      \\
      \label{e.err-A.2}
      \| \Afr^{(0)} \|_{\ell^2(\ZZ^d)} + \| \Afr^{(1)} - \Afr
      \|_{\ell^2(\ZZ^d)} + \sum_{m \geq 2} \| \Afr^{(m)} \|_{\ell^2(\ZZ^d)}
      &\leq C\exp\tond{-\frac{c}{\eps}}\ .
    \end{align}
  \item Let $z_{\textrm{ds}}(t) = \simone{T_\Chi^{-1}(\zeta_{\textrm{ds}})(t)}$
    and $z_{\textrm{br}}(t) = \simone{T_\Chi^{-1}(\zeta_{\textrm{br}})(t)}$ be the
    discrete soliton and the discrete breather solutions in the
    original coordinates, and $T_{\textrm{ds}}$ and $T_{\textrm{br}}$
    the corresponding periods; then it holds true
    \begin{equation}
      \label{e.cor.est}
      \sup_{|t|\leq \max\graff{T_{\textrm{ds}},T_{\textrm{br}}}}
      \norm{z_{\textrm{ds}}(t) - z_{\textrm{br}}(t)} \leq
      C\exp\tond{-\frac{c}\eps} \ .
    \end{equation}
  \end{enumerate}
\end{theorem}

We now assume a stronger condition than the invertibility of the
Jacobian operator $J_\Omega$; we require $\A$ to be a nondegenerate
extremizer for $\Zscr_1$ constrained to constant values of the norm
$G$. Under this assumption, which implies invertibility of $J_\Omega$,
it follows that for $\eps$ sufficiently small also the discrete
soliton $\Afr$ obtained in Proposition~4.1 is a nondegenerate
extremizer for $\Zscr:=\eps^{-1}Z$ constrained to the sphere
$\Sph:=\graff{G(z)=\nu}$, with $\nu$ sufficiently small (as required
by the normal form construction). As a consequence, $\Afr$ is an
orbitally stable periodic orbits (see \cite{Bam96,PalP16,Wei99}) for
the normal form $K=G+Z$. \simone{Once we add the remainder $P$ to $K$ we
  can't guarantee anymore that such a geometry is preserved, since $P$
  is not known to Poisson commute with the $\ell^2$ norm $G$. However,
  the smallness of $P$ in the norm \eqref{e.norms} ensures that the
  orbit remains in a prescribed region for long times.} We are going
to show that $\Afr$ is an approximate periodic orbit for the full
system $H=G+Z+P$ which is orbitally stable for exponentially long
times and that the same kind of stability holds true for the
Klein-Gordon breathers.

Let us introduce with $\bar{\Afr}:=\graff{\Afr^{(m)}}$ and denote with
$\Orb(\bar\Afr)$ the closed curve described by the Klein-Gordon
breather
\begin{displaymath}
  \Orb(\bar\Afr):=\graff{\simone{\zeta}_{\textrm{br}}(t),\,t\in[0,T]}
  \qquad\qquad
  \Orb:=T_\Chi^{-1}\Orb(\bar\Afr) \ .
\end{displaymath}
The next Theorem provides the orbital stability of $\Orb(\bar\Afr)$.

\begin{theorem}
  \label{t.exp.stab}
  \simone{Assume $\A$ to be a nondegenerate extremizer for $\Zscr_1$
  constrained to constant values of the norm $G$.} Let $z_0\in
  B_{R,3\dfr}$ with $R< 1$. Then $\forall\ 0<\mu\ll 1$,
  $\exists\ 0<\delta\ll 1$ such that
  \begin{equation}
    \label{e.lt.est}
    \inf_{w\in \Orb}\norm{z_0-w}< \delta
    \quad\Rightarrow\quad
    \inf_{w\in \Orb}\norm{\Phi^t_H(z_0)-w}<
    \mu\ ,\quad |t|< \exp\tond{\frac{c}\eps}\ .
  \end{equation}
\end{theorem}

\begin{remark}
In Theorems~\ref{t.hoa} and~\ref{t.exp.stab}, $c$ and $C$ are
suitable constants independent of $\eps$.
\end{remark}


\subsection{Proof of Theorem~\ref{t.main} (Normal Form Theorem)}
\label{ss:p1}

We give a sketch of the proof, which would be long and technical
if all the details were included. The estimates here included can be
obtained by following \cite{BGG89,BamG93}.

Recursive estimates, which are the most technical aspect of the whole
construction, need estimates on the initial size of the perturbation
$F$ and its vector field $X_F$. We thus introduce the main quantities
$E$ and $\omega_1$ providing the initial estimates
\begin{equation*}
  \nfun{F}{0}\leq E:= \eps\quadr{4d R^2 + \frac1{2p+2}R^{2p+2}}\ ,
  \qquad\quad
  \nvec{F}{0}\leq \omega_1:= 2\eps\quadr{C_d+R^{2p}}\ .
\end{equation*}






\begin{remark}
The magnitudes of $E$ and $\omega_1$ are coherent: since
$E=\Oscr(\eps R^2)$, then its differential, divided by $R$ according
to the definition of $\nvec{\cdot}{\dfr}$ in \eqref{e.norms}, has to
be $\Oscr(\eps)$.
\end{remark}



In order to solve \eqref{e.om.eq} we average along the periodic flow
$\Phi^t_G$ of period $2\pi$, as claimed by the following Lemma (for
the proof, see \cite{Bam96}):
\begin{lemma}
  \label{l.hom.eq}
  The homological equation~\eqref{e.om.eq}, i.e.
$
  \Poi{G}{\Chi} = \Psi(z) - Z(z)\ ,
$
  is solved by
  \begin{equation*}
    Z(z) = \frac1{2\pi}\int_0^{2\pi} \Psi\circ\Phi^t_G (z)dt\ ,
    \qquad\quad
    \Chi(z) =
    \frac1{2\pi}\int_0^{2\pi} t\quadr{\tond{\Psi-Z}\circ\Phi^t_G} (z)dt\ ;
  \end{equation*}
  for any $\dfr$ it satisfies the following estimates
  \begin{equation}
    \label{e.est.hom}
    \begin{aligned}
      \nfun{Z}{\dfr}&\leq \nfun{\Psi}{\dfr}\ ,
      &\qquad\qquad\quad
      \nvec{Z}{\dfr}&\leq \nvec{\Psi}{\dfr}\ ,
      \\
      \nfun{\Chi}{\dfr}&\leq 2\pi\nfun{\Psi}{\dfr}\ ,
      &\qquad\qquad\quad
      \nvec{\Chi}{\dfr}&\leq 2\pi\nvec{\Psi}{\dfr}\ .
      \end{aligned}
  \end{equation}
\end{lemma}
At first order, Lemma \ref{l.hom.eq} immediately provides the
estimates
\begin{equation}
  \label{e.est.step1}
  \nvec{Z_1}{\dfr}\leq\omega_1\ ,\qquad\qquad \nvec{\Chi_1}{\dfr}\leq
  \phi:=2\pi\omega_1\ ,
\end{equation}
which introduce the main perturbation parameter $\phi=\Oscr(\eps)$ of
the normal form scheme. Let now the arbitrary integer $r\geq 1$ be the
order of the normal form construction, i.e. the number of generating
functions $\Chi_s$ in the generating sequence
$\Chi=\graff{\Chi_s}_{s=1}^r$ and thus the number of homological
equations \eqref{e.om.eq} to be solved. The first important result
gives the bounds for the quantities involved in \eqref{e.om.eq}
\begin{lemma}
  \label{p.rec.est}
  Let $\dfr_s=\frac{s \dfr}{r}$, with $\dfr\leq \frac14$. Then, for any
  $1\leq s\leq r$ it holds true
  \begin{equation}
    \label{e.rec.fin}
      \nvec{\Theta}{\dfr_{s-1}} \leq
      \frac{\omega_1}{s}\tond{\frac{6r\phi}{\dfr}}^{s-1}
      \ \
      \forall \Theta\in\{F_s,\Psi_s,Z_s\}\ ,
      \quad
      \nvec{\Chi_s}{\dfr_{s-1}} \leq
      \frac{\phi}{s}\tond{\frac{6r\phi}{\dfr}}^{s-1} .
\end{equation}
\end{lemma}
The above Lemma, and in particular the last of \eqref{e.rec.fin},
allows to control the deformation of functions and vector fields
under the canonical transformation; indeed, let
  \begin{equation}
  \begin{aligned}
    T_\Chi f &=\sum_{r\geq 0} f_r \ ,
    &\qquad
    f_r &:=\sum_{j=1}^r\frac{j}{r} L_{\Chi_j}f_{r-j} \ ,
    &\qquad
    f_0 &=f\ ,
    \\
    T_\Chi g &=\sum_{r\geq 1}g_r\ ,
    &\qquad
    g_r &:=\sum_{j=1}^{r-1}\frac{j}{r-1}L_{\Chi_j}g_{r-j}\ ,
    &\qquad
    g_1 &=g\ ,
  \end{aligned}
\end{equation}
then the following holds true
\begin{lemma}
  \label{l.10.3}
Let us introduce $M_1<M_2<1$
\begin{equation}
\label{e.M1.M2}
M_1(\phi,r):=\frac{\phi(e+3r)}{\dfr}\ ,\qquad\qquad
M_2(\phi,r):=\frac{\phi(2e+3r)}{\dfr}\ .
\end{equation}
Then, for any $\dfr\leq\frac14$ and any $r\geq 1$ the following bounds
hold true
\begin{equation}
  \label{e.10.3.vec}
  \begin{aligned}
    \nfun{z_r}{\dfr}&\leq M_1^{r-1}\tilde G_1\ ,
    &\qquad
    \nfun{f_r}{\dfr}&\leq M_1^{r-1} \tilde B_1\ ,
    &\qquad
    \nfun{g_r}{\dfr}&\leq M_1^{r-2} \tilde\Gamma_2\ ,
    \\
    \null&\null
    &\qquad
    \nvec{f_r}{\dfr}&\leq M_2^{r-1} \bar B_1\ ,
    &\qquad
    \nvec{g_r}{\dfr}&\leq M_1^{r-2}\bar\Gamma_2\ ,
  \end{aligned}
\end{equation}
together with
\begin{equation*}
    \tilde G_1:=\frac{R\phi}{\dfr},
    \quad
    \tilde B_1:=\frac\phi{\dfr}\nfun{f}{0},
    \quad
    \tilde\Gamma_2:=\frac\phi{\dfr}\nfun{g}{0},
    \quad
    \bar B_1:=\frac{2\phi}{\dfr} \nvec{f}{0},
    \quad
    \bar\Gamma_2:=\frac{2\phi}\dfr \nvec{g}{0},
\end{equation*}
\end{lemma}

Given the above estimates, we can obtain the inclusions
\eqref{e.incl.balls}; indeed, we have
\begin{equation*}
\nfun{T_\Chi z - z}{\dfr}\leq \sum_{r\geq 1}\nfun{z_r}{\dfr}\leq
\frac{\tilde G_1}{1-M_1}<2\tilde G_1 = \phi\tond{\frac{2R}{\dfr}}\ ,
\end{equation*}
provided we ask for $M_1<\frac12$; thus the deformation is
$\Oscr(R\eps)$. The remainder $\Rscr^{(r+1)}$ in \eqref{e.NF}, at an
arbitrary step $r$, is given by
\begin{displaymath}
\Rscr^{(r+1)} = \sum_{s\geq r+1}G_s + \sum_{s\geq r+1}F_s \ ;
\end{displaymath}
by exploiting \eqref{e.10.3.vec} and the initial estimates, the
following bounds hold true
\begin{equation}
  \label{e.rem.est}
  \nvec{\Theta}{\dfr} \leq \tond{\frac{2\phi}{\dfr}} M_2^{s-1} ,
  \ \forall\Theta\in\{G_s,F_s\}
  \quad\Rightarrow\quad
\nvec{\Rscr^{(r+1)}}{\dfr}\leq
\tond{\frac{2\phi}{\dfr}}\frac{M_2^r}{1-M_2}\ .
\end{equation}
The exponential estimate \eqref{e.main.est.1} is derived from
\eqref{e.rem.est} by expanding $M_2^r$ as
\begin{equation}
  \label{e.exp.est}
M_2^r = \tond{\frac{6\phi}{\dfr}}^r r^r \tond{1+\frac{1}{r}}^r <
e\tond{\frac{6\phi}{\dfr}}^r r^r\ ,
\end{equation}
and optimizing the number of normal form steps
$r=r_{opt}:=\lfloor\frac{\dfr}{6e\phi}\rfloor$; in this way $r$ is
related $\eps$.  Finally, the variation of $G$ along the generic orbit
in the neighbourhood of the origin is obtained combining the estimate
of the Poisson bracket
\begin{displaymath}
|G(t)-G(0)|\leq \int_0^t|\Poi{G}{\Rscr^{(r+1)}}(z(s))|ds\leq
|t|\nfun{G}{\dfr}\tond{\frac{4e\phi}{\dfr^2}}\exp\quadr{-\tond{\frac{\dfr}{6e\phi}}}\ ,
\end{displaymath}
with the bound on the deformation
\simone{$|G(z)-G(T_\Chi(z))|=|(T_\Chi G-G)(z)|$}
and exploiting the fact that $G$ coincides with the norm on the phase
space. Similarly one can control the variation of $Z$.\qed


\subsection{Proof of Theorem~\ref{t.hoa} (high order approximation)}
\label{ss:p2}

The proof of point $\emph{(1)}$ is an easy application of the Implicit
Function Theorem (I.F.T.), based on the main assumption that $J_\Omega$ is
invertible in the subspace of real square-summable sequences. Indeed
$\Afr=\mathcal{A}$ is the ``unperturbed'' solution of equation
\eqref{e.dnls.1}
$
f(\mathcal{A},\Omega,0)=f_0(\mathcal{A},\Omega)=0 \ ,
$
and the Jacobian of $f$ evaluated at $\tond{\A,\eps=0}$ is given by
$
D_\Afr f(\mathcal{A},\Omega,0)=J_\Omega\ ,
$
which is invertible. The first of
\eqref{e.ds.per.bound} is standard, while the second can be easily
obtained from
\begin{displaymath}
D_{\Afr}f(\Afr) = D_\Afr f_0(\Afr) + \Oscr(\eps) = D_\Afr
f_0(\mathcal{A}) + D^2_\Afr
f_0(\mathcal{A})\tond{\Afr-\mathcal{A}}+\Oscr(\eps) = J_\Omega + \Oscr(\eps)\ .
\end{displaymath}

Once proved the existence of $\Afr$, discrete soliton of $K$, we
follow the same strategy used for Theorem~\eqref{theorem-existence} in
order to prove the existence of an analytic branch of Klein-Gordon
breathers\footnote{We again ``identify'' solitons and breathers with
  thier amplitude(s).} $\bar\Afr$ exponentially close to
$\Afr(\Omega,\eps)$.  Let us consider the Hamilton equation for
\eqref{e.NF} restricting to the variable $\zeta$ only
(recall $i\eta=-\bar\zeta$)
\begin{displaymath}
\omega\partial_\tau\zeta = \im\zeta + \partial_\eta Z + \partial_\eta
\Rscr^{(r+1)}\ ;
\end{displaymath}
by inserting the Fourier expansion \eqref{e.br.fourier} with real
coefficients $\Afr^{(m)}$, we get the equation
\begin{displaymath}
\hat L_\omega \bar{\Afr} + \partial_\eta Z(\bar{\Afr}) + \partial_\eta
\Rscr^{(r+1)}(\bar{\Afr}) = 0\ ,
\qquad\qquad
\tond{\hat L_\omega\bar{\Afr}}^{(m)} := \im(1-m\omega)\Afr^{(m)}\ ,
\end{displaymath}
since the Hamilton equations are first order in time. The Kernel $\hat
V_2$ of the linear operator $\hat L_{\omega=1}$ is given only by
$e^{\im\tau}$, all the other harmonics belonging to the Range $\hat
W_2$. We decompose $\bar{\Afr} = \bar\Afr^\sharp + \bar\Afr^\flat$ and project
the equations on $\hat V_0$ and $\hat W_0$. The great difference with
respect to the proof of Theorem~\ref{theorem-existence}, is that the
resonant normal form construction, performed averaging with respect to
the periodic flow $e^{\im\tau}$, provides a natural decomposition of
the term $\partial_\eta Z(\bar\Afr^\sharp)$, so that one has for free
\begin{equation}
\label{e.proj.main}
\Pi_{\hat W_0}\partial_\eta Z(\bar\Afr^\sharp) = 0\ ,\qquad\qquad
\Pi_{\hat V_0}\partial_\eta Z(\bar\Afr^\sharp) = \partial_\eta Z(\bar\Afr^\sharp)\ .
\end{equation}
This remarkable property allows to write the Range equation as
\begin{displaymath}
  \hat L_\omega \bar\Afr^\flat + \quadr{
    \Pi_{\hat W_0}\partial_\eta Z(\bar\Afr^\flat+\bar\Afr^\sharp) -
    \Pi_{\hat W_0}\partial_\eta Z(\bar\Afr^\sharp)} +
  \Pi_{\hat W_0}\partial_\eta \Rscr^{(r+1)}(\bar\Afr^\flat+\bar\Afr^\sharp)=0\ ,
\end{displaymath}
where the term in square brackets
is at least linear in $\bar\Afr^\flat$, $\Oscr(\eps\bar\Afr^\flat)$,
hence a small perturbation with respect to $\hat L_{\omega=1}$. The
usual leading order approximation provided by the I.F.T. gives, for
some constant $C$ independent of $\eps$,
\begin{equation}
\label{e.range.est.1}
\bar\Afr^\flat \approx -\tond{\hat L_\omega}^{-1}\Pi_{\hat
  W_0}\partial_\eta \Rscr^{(r+1)}(\bar\Afr^\flat) \qquad\Rightarrow\qquad
\norm{\bar\Afr^\flat}\leq C\eps^{r+1}\ .
\end{equation}

The Kernel equation, after dividing by $\eps$, takes the form
\begin{displaymath}
\Omega\bar\Afr^\sharp - \Delta\bar\Afr^\sharp +
\Gamma_p|\bar\Afr^\sharp|^{2p}\bar\Afr^\sharp +
\sum_{s=2}^{r}\eps^{s-1}\Pi_{\hat V_0}\partial_\eta
\Zscr_s\tond{\bar\Afr^\sharp)} + \Oscr\tond{\eps^r}=0\ ,
\end{displaymath}
where we have introduced the scaled functions
$\Zscr_s:=\eps^{-s}Z_s$. Notice that in the small remainder we
included not only the smallness of the vector field
$\partial_\eta \Rscr^{(r+1)}$, but also
\begin{displaymath}
{\Pi_{\hat V_0}\partial_\eta \Zscr_s(\bar\Afr^\flat+\bar\Afr^\sharp) - \Pi_{\hat
    V_0}\partial_\eta \Zscr_s(\bar\Afr^\sharp)} = \Pi_{\hat V_0}\partial_\eta
\Zscr_s(\bar\Afr^\flat+\bar\Afr^\sharp) - \partial_\eta \Zscr_s(\bar\Afr^\sharp)=
\Oscr(\bar\Afr^\flat)\ ,
\end{displaymath}
which is of order $\Oscr\tond{\eps^{r+1}}$ for any $1\leq s\leq r$,
because of \eqref{e.range.est.1}. The Kernel equation now turns out to
be a perturbation of order $\Oscr(\eps^{r})$ of
\begin{equation}
\label{e.ker.tr}
f(\bar\Afr^\sharp,\Omega,\eps):=\Omega\bar\Afr^\sharp - \Delta\bar\Afr^\sharp +
\Gamma_p|\bar\Afr^\sharp|^{2p}\bar\Afr^\sharp +
\sum_{s=2}^{r}\eps^{s-1}\Pi_{\hat V_0}\partial_\eta
\Zscr_s\tond{\bar\Afr^\sharp)}=0\ ,
\end{equation}
since the last term in the sum of $\Zscr_s$ is of order
$\Oscr(\eps^{r-1})$.  Equation \eqref{e.ker.tr} admits the solution
$\bar\Afr=\bar\Afr(\Omega,\eps)$ given by Proposition~4.1., and since
$J_{\Omega,\eps}$ is invertible for $\eps$ sufficiently small, a fixed
point argument (see for example Appendix of \cite{BamPP10}) can be
used to conclude the proof; estimate \eqref{e.err-omega.2} is
standard, while estimate \eqref{e.err-A.2} follows once the generic
step $r$ is replaced with the optimal choice $r=r_{opt}(\eps)$.

In order to prove point $\emph{(3)}$ we exploit the fact that the two
periods, $T_{\textrm{ds}}$ and $T_{\textrm{br}}$, are both
approximately equal to $2\pi$, because of \eqref{e.err-omega.2}. Hence
\begin{displaymath}
\norm{\zeta_{\textrm{ds}}(t) - \zeta_{\textrm{br}}(t)}\leq \sum_{m\neq
  1}\norm{\Afr^{(m)}} + \norm{\Afr^{(1)}e^{\im\omega t} - \Afr
  e^{\im(1-\frac\eps2\Omega)t}}\ ;
\end{displaymath}
moreover, on a time interval of order
$\Oscr(1)=\max\graff{T_{\textrm{ds}},T_{\textrm{br}}}$, we have
\begin{displaymath}
\norm{\Afr^{(1)} e^{\im\omega t} - \Afr
  e^{\im(1-\frac\eps2\Omega)t}}\leq \norm{\Afr^{(1)}
  e^{\Oscr(\eps^2)t} - \Afr}\leq C\norm{\Afr^{(1)} - \Afr}\ .
\end{displaymath}
The estimates holds true also in the original coordinates, exploiting
the Lipschitz continuity of the canonical transformation
$T^{-1}_\Chi$, with a Lipschitz constant $L=\Oscr(1)$. \qed


\subsection{Proof of Theorem~\ref{t.exp.stab}
  (exponentially long time stability)}
\label{ss:p3}

We collect the main geometrical ideas already exploited in
\cite{Bam96,PalP16}, omitting most of the details that the interested
reader can find in the quoted papers.

We first have to prove the orbital stability of the discrete soliton
$\Afr$, interpreted as approximate breather solution of the
Klein-Gordon model. We denote with $\Orb(\Afr)$ the closed orbit
described by the discrete soliton profile $\Afr$ during its periodic
evolution
\begin{displaymath}
\Orb(\Afr):=\graff{e^{\im
    (1-\frac\eps2\Omega)t}\Afr,\,t\in[0,T_{\textrm{ds}}]}\ .
\end{displaymath}
We consider a tubolar neighbourhood $\mathcal{W}_0$ of $\Orb(\Afr)$,
in the transformed coordinates which give the original Hamiltonian the
normal form \eqref{e.main.est.1}. Any point $z\in\mathcal{W}_0$ can be
represented with a local set of coordinates
\begin{displaymath}
z = (\ph,E,v)\in\RR\times\RR\times V_\xi\ ,
\end{displaymath}
where $\xi$ is the projection of $z$ on $\Orb(\Afr)$ and $V_\xi$ is
the orthogonal complement to the symmetry field $X_G(\xi)$ in the
tangent space $T_\xi \Sph = V_\xi\oplus X_G(\xi)$. The three
coordinates represent, respectively: the scalar coordinate $\ph$ is
the tangential displacement along the field $X_G$, the scalar
coordinate $E$ is the displacement in the direction $\nabla G$
orthogonal to the surface $\Sph$ and the vector valued coordinate $v$
is the tangential displacement in the directions of $V_\xi$. In order
to measure the orbital distance of a generic point $z$ from
$\Orb(\Afr)$ we need to control only the directions transversal to the
orbits, hence $E$ and $v$. The main point is that $E$ is related to
the variation of $G$ while $v$ is related to the variation of $\Zscr$:
the first is obvious, while the second holds because we have asked
$\Afr$ to be a nondegenerate extremizer of $\Zscr$ on $\Sph$, hence
locally we have
\begin{displaymath}
\norm{v}^2 = \norm{z-\xi}^2\leq \frac1{C}|\Zscr(z)-\Zscr(\xi)|\ ,
\end{displaymath}
where $C$ is a constant depending on $\Zscr''(\xi)$. Let us consider
an initial datum $z_0\in\mathcal{W}_0$ and its piece of orbit
$\Phi^t_H(z_0)\cap \mathcal{W}_0$; for any point on this curve we have
\begin{align*}
  \inf_{w\in \Orb(\Afr)\cap\mathcal{W}_0}\norm{\Phi^t_H(z_0) - w}
  &\leq
  c_1|E(t)| + c_2\norm{v(t)}
  \leq
  \\
  &\leq
  C\sqrt{|G(\Phi^t_H(z_0)) - G(\Afr)| + |\Zscr(\Phi^t_H(z_0)) - \Zscr(\Afr)|}\ .
\end{align*}
If $z_0$ is taken in a suitable domain where \eqref{e.main.est.2} hold
true, then the two terms in the square root can be bounded by
\begin{align}
  |G(\Phi^t_H(z_0)) - G(\Afr)| &\leq |G(\Phi^t_H(z_0)) - G(z_0)| + |G(z_0) - G(\Afr)|\\
  |\Zscr(\Phi^t_H(z_0)) - \Zscr(\Afr)| &\leq |\Zscr(\Phi^t_H(z_0)) -
  |\Zscr(z_0)|+|\Zscr(z_0) - \Zscr(\Afr)|\ ;
\end{align}
the first right hand terms are exactly controlled by
\eqref{e.main.est.2} on exponentially long times, while the second
right hand terms are controlled by the initial distance from the
orbit. This allows to get \eqref{e.lt.est} for the discrete soliton
$\Afr$ in the normal form coordinates.  The stability result then is
transferred to the \simone{orbit $\Orb(\bar\Afr)$ of the} Klein-Gordon
breather exploiting the exponentially small distance between the two
orbits, as stated by \eqref{e.cor.est}. Finally, the same stability
result holds also in the original coordinates, since $T_\Chi^{-1}$ is
Lipschitz with a Lipschitz constant $L=\Oscr(1)$.\qed


\end{document}